\documentclass[a5paper,10pt,reqno]{amsart}

\usepackage{amsmath,amssymb,amsthm}
\usepackage{latexsym}
\usepackage{color}
\usepackage{graphicx}
\usepackage{mathrsfs}
\usepackage{enumerate}
\usepackage[abbrev]{amsrefs}
\usepackage[T1]{fontenc}

\makeatletter

\@addtoreset{equation}{section}
\makeatother
\allowdisplaybreaks[4]

\theoremstyle{plain}
\newtheorem{thm}{Theorem}[section]
\newtheorem{lemm}[thm]{Lemma}
\newtheorem{prop}[thm]{Proposition}
\newtheorem{cor}[thm]{Corollary}

\theoremstyle{definition}

\newtheorem{rem}[thm]{Remark}

\newcommand{\h}{{\rm h}}
\newcommand{\nablah}{\nabla_{{\rm h}}}
\newcommand{\Deltah}{\Delta_{{\rm h}}}

\newcommand{\xh}{x_{\rm h}}
\newcommand{\yh}{y_{\rm h}}
\newcommand{\xih}{\xi_{\rm h}}

\newcommand{\vv}{{\rm v}}

\newcommand{\alphah}{\alpha_{\rm h}}
\newcommand{\alphav}{\alpha_3}
\newcommand{\Lh}{L_{\rm h}}
\newcommand{\Lv}{L_{\rm v}}
\newcommand{\Gh}{G_{\rm h}}
\newcommand{\Gv}{G_{\rm v}}
\newcommand{\uh}{u_{\rm h}}
\newcommand{\vh}{v_{\rm h}}
\newcommand{\Nhh}{\mathcal{D}^{\rm h}}
\newcommand{\Nvv}{\mathcal{D}^{\rm v}}

\begin{document}
\title[$3$D Anisotropic Navier-Stokes equation]
{Large time behavior of solutions to the $3$D anisotropic Navier-Stokes equation}
\author[Mikihiro Fujii]{Mikihiro Fujii}
\address[]{Graduate School of Mathematics Kyushu University,Fukuoka 819--0395, JAPAN}
\email{3MA20005M@s.kyushu-u.ac.jp}
\keywords{$3$D anisotropic Navier-Stokes equation, large time behavior, decay estimates, asymptotic profile}
\subjclass[2010]{35Q30, 35B40, 35Q35}
\begin{abstract}
	We consider the large time behavior of the solution to the $3$D Navier-Stokes equation with horizontal viscosity $\Delta_{\rm h} u=\partial_1^2 u+\partial_2^2 u$ and show that the $L^p$ decay rate of the horizontal components of the velocity field coincides to that of the $2$D heat kernel, while the vertical component decays like the $3$D heat kernel.
	Moreover, we consider the asymptotic expansion of the solution and find that a portion of the nonlinear term affect the leading term of the horizontal components of the velocity field, whereas the leading term of the vertical component is given by only the linear solution.
\end{abstract}
\maketitle

\section{Introduction}\label{s1}
In this paper, we consider the initial value problem for the $3$D anisotropic Navier-Stokes equation:
\begin{align}\label{eq:ANS_1}
	\begin{cases}
		\partial_tu-\Deltah u+(u\cdot \nabla)u+\nabla p=0, & \qquad t>0,x\in \mathbb{R}^3,\\
		\nabla \cdot u=0, & \qquad t\geqslant 0,x\in \mathbb{R}^3,\\
		u(0,x)=u_0(x), & \qquad x\in \mathbb{R}^3.
	\end{cases}
\end{align}
Here, $u=(u_1(t,x),u_2(t,x),u_3(t,x))$ and $p=p(t,x)$ denote the unknown velocity and the unknown pressure of the fluid, respectively,
while the vector field $u_0=(u_{0,1}(x),u_{0,2}(x),u_{0,3}(x))$ is the given divergence free initial velocity of the fluid.
The operator $\Deltah:=\partial_1^2+\partial_2^2$ denotes the horizontal Laplacian and $\nabla=(\partial_1,\partial_2,\partial_3)$ represents the $3$D gradient.
Throughout this paper, for given $3$D vector $a=(a_1,a_2,a_3)\in \mathbb{R}^3$, we write $a_{\h}=(a_1,a_2)$.
We also denote by $\nablah=(\partial_1,\partial_2)$ the horizontal gradient.

In geophysical fluid dynamics,
meteorologists modelize the turbulent diffusion with anisotropic viscosity
$-\nu\Deltah-\varepsilon\partial_3^2$,
where the horizontal kinetic viscosity coefficient $\nu$ and the vertical kinetic viscosity coefficient $\varepsilon$ satisfy $0<\varepsilon \ll \nu$.
We refer to \cite[Chapter 4]{Pedlosky} for the complete discussion
for the physical background.
In this article, we are concerned with the system (\ref{eq:ANS_1}), which corresponds to the case $\nu=1$ and $\varepsilon=0$.

The aim of this paper is to reveal the anisotropic effect for the large time behavior of the solution $u$ to (\ref{eq:ANS_1})
and we derive the $L^p$ decay rate and the asymptotic expansion of the solution.
More precisely, for $s\in \mathbb{N}$ with $s\geqslant5$ and for sufficiently small initial data
$u_0\in H^s(\mathbb{R}^3)\cap L^1(\mathbb{R}_{\xh}^2;(W^{1,1}\cap W^{1,\infty})(\mathbb{R}_{x_3}))$ with $\nabla\cdot u_0=0$,
the solution $u$ to (\ref{eq:ANS_1}) satisfies
\begin{align*}
	\|\nabla^{\alpha}\uh(t)\|_{L^p}=O(t^{-(1- \frac{1}{p})-\frac{|\alphah|}{2}}),\qquad
	\|\nablah^{\alphah} u_3(t)\|_{L^p}=O(t^{-\frac{3}{2}(1- \frac{1}{p})- \frac{|\alphah|}{2}})
\end{align*}
as $t\to \infty$ for $1\leqslant p\leqslant \infty$ and $\alpha=(\alphah,\alpha_3)\in (\mathbb{N}\cup\{0\})^2\times(\mathbb{N}\cup\{0\})$ with $|\alpha|\leqslant 1$.
This implies that $\uh(t)$ decays like the $2$D heat kernel, whereas
$u_3(t)$ decays as the $3$D heat kernel.
Moreover, we shall show that $\uh(t)$ and $u_3(t)$ behave as
\begin{align*}
	\uh(t,x)&= \Gh(t,\xh)\int_{\mathbb{R}^2}u_{0,\h}(\yh,x_3)d\yh\\
				&\quad-\Gh(t,\xh)\int_0^{\infty}\int_{\mathbb{R}^2}\partial_3(u_3\uh)(\tau,\yh,x_3)d\yh d\tau+o(t^{-(1- \frac{1}{p})})\quad {\rm in\ }L^p(\mathbb{R}^3)\quad (1\leqslant p \leqslant \infty),\\
	u_3(t,x)&= \Gh(t,\xh)\int_{\mathbb{R}^2}u_{0,3}(\yh,x_3)d\yh+o(t^{-\frac{3}{2}(1- \frac{1}{p})})\quad {\rm in\ }L^p(\mathbb{R}^3)\quad (1\leqslant p < \infty),
\end{align*}
as $t\to \infty$,
where $\Gh(t,\xh)$ denotes the $2$D Gaussian.
Furthermore, we also prove that if, in addition, $s\geqslant 9$ and
$|\xh|u_0(x)\in L^1(\mathbb{R}_{\xh}^2;(L^{1}\cap L^{\infty})(\mathbb{R}_{x_3}))$,
then the remainder terms $o(t^{-(1- \frac{1}{p})})$ and $o(t^{-\frac{3}{2}(1- \frac{1}{p})})$
of the above asymptotic expansions are improved to $O(t^{-(1- \frac{1}{p})-\frac{1}{2}}\log t)$ and $O(t^{-\frac{3}{2}(1- \frac{1}{p})- \frac{1}{2p}})$ for $1<p\leqslant \infty$, respectively
and also we obtain the higher order expansion of the vertical component of the solution.

Before we state our main theorems, let us recall known results related to the system (\ref{eq:ANS_1}).
For the well-posedness of (\ref{eq:ANS_1}), Chemin, Desjardins, Gallagher and Grenier \cite{CDGG} proved
the existence of a local solution for large data and a global solution for small data in $H^{0,s}(\mathbb{R}^3)$ ($s>1/2$), where
\begin{align*}
	H^{\sigma,s}(\mathbb{R}^3):=L^2(\mathbb{R}^3)\cap \dot{H}^{\sigma,s}(\mathbb{R}^3),\qquad
	\dot{H}^{\sigma,s}(\mathbb{R}^3):=(-\Deltah)^{-\frac{\sigma}{2}}(-\partial_3^2)^{-\frac{s}{2}}L^2(\mathbb{R}^3).
\end{align*}
The uniqueness of solutions in $H^{0,s}(\mathbb{R}^3)$ ($s>1/2$) is proved by Iftimie \cite{Iftimie}.
Here, the regularity condition $s>1/2$ is caused by the Sobolev embedding $H^s(\mathbb{R}_{x_3})\hookrightarrow L^{\infty}(\mathbb{R}_{x_3})$
and $s=1/2$ corresponds to the scaling critical exponent.
Paicu \cite{Paicu} considered the scaling critical setting
and proved the global existence of a unique solution in the $L^2$-based anisotropic Besov space $\mathcal{B}^{0,\frac{1}{2}}(\mathbb{R}^3)$.
Chemin and Zhang \cite{CZ} and Zhang and Fang \cite{ZF} extend the Paicu theorem to the $L^p$ framework and proved the global existence results in the scaling critical anisotropic Besov space $\mathcal{B}^{-1+\frac{2}{p},\frac{1}{2}}_{p}(\mathbb{R}^3)$ ($2\leqslant p< \infty$).
We refer to \cite{BCD,LPZ,PM,YY,ZT} for other literatures on the well-posedness for (\ref{eq:ANS_1}).

Next, we forcus on the previous studies for the large time behavior.
Ji, Wu and Yang \cite{JWY} proved that for given small initial data $u_0\in H^4(\mathbb{R}^3)\cap H^{-\sigma,1}(\mathbb{R}^3)$ ($3/4\leqslant\sigma<1$),
the global solution $u$ of (\ref{eq:ANS_1}) satisfies
\begin{align}\label{JWY_decay}
	\|u(t)\|_{H^4\cap H^{-\sigma,1}}\leqslant C\|u_0\|_{H^4\cap H^{-\sigma,1}},\qquad
	\|\nabla^{\alpha} u(t)\|_{L^2}\leqslant C(1+t)^{-\frac{\sigma+|\alphah|}{2}}\|u_0\|_{H^4\cap H^{-\sigma,1}}
\end{align}
for $t\geqslant 0$ and $\alpha=(\alphah,\alpha_3)\in (\mathbb{N}\cup\{0\})^2\times (\mathbb{N}\cup\{0\})$ with $|\alpha|\leqslant 1$.
The decay rate of (\ref{JWY_decay}) coincides to that of the linear solution $e^{t\Deltah}u_0$.
Xu and Zhang \cite{XZ} relaxed the regularity condition of \cite{JWY} and showed that
for $s>2$ and $(1+3s)/\{10(s-1)\}<\sigma<1$, the solution $u$ of (\ref{eq:ANS_1})
with given small initial data $u_0\in (\dot{H}^{0,s}\cap\dot{H}^{-\sigma,0}\cap \dot{H}^{-\sigma,-\frac{\sigma}{2}-\frac{1}{4}}\cap \dot{H}^{-\frac{1}{2},1})(\mathbb{R}^3)$ satisfies that
for $\alphah\in (\mathbb{N}\cup\{0\})^2$ with $|\alphah|\leqslant 1$,
\begin{align}\label{XZ_decay}
	\begin{split}
		&\|\nablah^{\alphah}u(t)\|_{L^2}=O(t^{-\frac{\sigma+|\alphah|}{2}}),\qquad
		\|\partial_3u(t)\|_{L^2}=O(t^{-\frac{1}{4}}),\\
		&\|\nablah^{\alphah}u_3(t)\|_{L^2}=O(t^{-\frac{1}{2}(\frac{3}{2}\sigma+\frac{1}{4}+|\alphah|)})
	\end{split}
\end{align}
as $t\to \infty$.
This implies that the horizontal components decay like the $2$D heat kernel and the vertical component decays as the $3$D heat kernel.
We refer to \cite{CJ,FM,HZ,KO} for the large time behavior of isotropic Navier-Stokes equations.

The purpose of this paper is to refine the results obtained by \cite{JWY,XZ}
and to clarify the effect of the anisotropy on the large time behavior of the solution
in terms of $L^p$ decay rates and asymptotic expansions.

In order to state our results precisely,
we prepare some notation.
For $s \in \mathbb{N}$, we define a function space $X^s(\mathbb{R}^3)$ by
\begin{align*}
	X^s(\mathbb{R}^3)
	&:=H^s(\mathbb{R}^3)\cap L^1(\mathbb{R}^2_{\xh};(W^{1,1}\cap W^{1,\infty})(\mathbb{R}_{x_3})).
\end{align*}
Let $\Gh(t,\xh)$ be the $2$D Gaussian:
\begin{align*}
	\Gh(t,\xh):=(4\pi t)^{-1}e^{-\frac{|\xh|^2}{4t}}, \qquad (t, \xh)=(t,x_1,x_2)\in (0,\infty)\times \mathbb{R}^2.
\end{align*}

Our first main result reads as follows.
\begin{thm}\label{thm:main_1}
	Let $s\in \mathbb{N}$ satisfy $s\geqslant 5$.
	Then, there exists a positive constant $\delta_1=\delta_1(s)$ such that
	the following properties hold:

	For any $u_0\in X^s(\mathbb{R}^3)$ satisfying $\nabla\cdot u_0=0$ and
	$\|u_0\|_{X^s}\leqslant \delta_1$, there exists a unique solution $u\in C([0,\infty);X^s(\mathbb{R}^3))$ of (\ref{eq:ANS_1}) and
	there exists an absolute positive constant $C$ such that
	\begin{align}
		&\|\nabla^{\alpha}\uh(t)\|_{L^p}\leqslant Ct^{-(1-\frac{1}{p})- \frac{|\alphah|}{2}}\|u_0\|_{X^s}, \label{main_1:uhL^p_decay}\\
		&\|\nablah^{\alphah}u_3(t)\|_{L^p}\leqslant Ct^{-\frac{3}{2}(1-\frac{1}{p})- \frac{|\alphah|}{2}}\|u_0\|_{X^s}\label{main_1:u3L^p_decay}
	\end{align}
	for all $1\leqslant p\leqslant \infty$, $t>0$ and $\alpha=(\alphah,\alpha_3)\in (\mathbb{N}\cup\{0\})^2\times (\mathbb{N}\cup\{0\})$ with $|\alpha|\leqslant 1$.

	Moreover, there hold for $1\leqslant p\leqslant \infty$,
	\begin{align}\label{main_1:asymp_uh}
		\begin{split}
			\lim_{t\to \infty}
			t^{1- \frac{1}{p}}
			&\left\|
				\uh(t,x)
				-
				\Gh(t,\xh)\int_{\mathbb{R}^2}u_{0,\h}(\yh,x_3)d\yh\right.\\
			&\left.
				\qquad \qquad+
				\Gh(t,\xh)\int_0^{\infty}\int_{\mathbb{R}^2}\partial_3(u_3\uh)(\tau,\yh,x_3)d\yh d\tau
			\right\|_{L^p_x}=0
		\end{split}
	\end{align}
	and for $1\leqslant p<\infty$,
	\begin{align}\label{main_1:asymp_u3}
		\lim_{t\to \infty}
		t^{\frac{3}{2}(1- \frac{1}{p})}
		\left\|
			u_3(t,x)
			-
			\Gh(t,\xh)\int_{\mathbb{R}^2}u_{0,3}(\yh,x_3)d\yh
		\right\|_{L^p_x}=0.
	\end{align}
\end{thm}
\begin{rem}\label{rem:main_1}
	\begin{itemize}
		\item [(1)] The decay rate of (\ref{main_1:uhL^p_decay})and (\ref{main_1:u3L^p_decay}) with $p=2$
		is same as that of the $L^2$-norm of the $2$D Gaussian,
		which corresponds to the limiting case $\sigma=1$ in (\ref{JWY_decay}) and (\ref{XZ_decay}).
		Thus, our result extends the previous studies \cite{JWY,XZ}.
		\item [(2)]
		Unlike isotropic nonlinear parabolic-type equations \cite{EZ,FM,IK,Iwabuchi,Kato,NSU,NY}, (\ref{main_1:asymp_uh})
		implies the leading term of $\uh(t)$ cannot be given by only the linear solution,
		and the nonlinearity affect the leading term.
		\item [(3)] The asymptotic limit (\ref{main_1:asymp_u3}) fails if $p=\infty$.
		Indeed, there exists an initial data such that the limit  (\ref{main_1:asymp_u3}) with $p=\infty$ does not converge to $0$.
		These facts are precisely stated in our second main result and its corollary.
	\end{itemize}
\end{rem}

If we suppose the spatial decay assumption on the initial data, then we obtain the convergence rate for the above limits (\ref{main_1:asymp_uh})-(\ref{main_1:asymp_u3})
and the higher order asymptotic expansion of the vertical component of the velocity field.
The following theorem is our second main result.
\begin{thm}\label{thm:main_2}
	Let $s\in \mathbb{N}$ satisfy $s\geqslant 9$.
	Then, there exists a positive constant $\delta_2=\delta_2(s)\leqslant \delta_1(s)$ such that
	if $u_0\in X^s(\mathbb{R}^3)$ satisfies
	$|\xh|u_0(x)\in L^1(\mathbb{R}^2_{\xh};(L^1\cap L^{\infty})(\mathbb{R}_{x_3}))$,
	$\nabla\cdot u_0=0$
	and $\|u_0\|_{X^s}\leqslant \delta_2$,
	then for $1< p\leqslant \infty$
	there exists a positive constant $C=C(p)$ such that
	\begin{align}\label{main_2:asymp_uh}
		\begin{split}
			&\left\|
				\uh(t,x)
				-
				\Gh(t,\xh)\int_{\mathbb{R}^2}u_{0,\h}(\yh,x_3)d\yh\right.\\
			&\left.
				\qquad \qquad+
				\Gh(t,\xh)\int_0^{\infty}\int_{\mathbb{R}^2}\partial_3(u_3\uh)(\tau,\yh,x_3)d\yh d\tau
			\right\|_{L^p_x}
			\leqslant
			C\|u_0\|_{\widetilde{X^s}}t^{-(1- \frac{1}{p})- \frac{1}{2}}\log t
		\end{split}
	\end{align}
	for $t\geqslant 2$
	and for $1\leqslant p\leqslant \infty$
	there exists a positive constant $C =C(p)$ such that
	\begin{align}\label{main_2:asymp_u3}
		\begin{split}
			&\left\|
				u_3(t,x)
				-
				\Gh(t,\xh)\int_{\mathbb{R}^2}u_{0,3}(\yh,x_3)d\yh
			\right\|_{L^p_x}
			\leqslant
			\begin{cases}
				C\|u_0\|_{\widetilde{X^s}}t^{-\frac{3}{2}(1- \frac{1}{p})- \frac{1}{2p}} & (1<p\leqslant \infty)\\
				C\|u_0\|_{\widetilde{X^s}}t^{- \frac{1}{2}}\log t & (p=1)
			\end{cases}
		\end{split}
	\end{align}
	for all $t\geqslant 2$,
	where
	$
	\|u_0\|_{\widetilde{X^s}}
	:=\|u_0\|_{X^s}+\||\xh|u_0(x)\|_{L^1(\mathbb{R}^2_{\xh};(L^1\cap L^{\infty})(\mathbb{R}_{x_3}))}
	$.

	Furthermore, for any $1< p\leqslant \infty$, it holds
	\begin{align}\label{main_2:asymp_u3_second}
		\begin{split}
			\lim_{t\to \infty}
			t^{\frac{3}{2}(1- \frac{1}{p})+\frac{1}{2p}}
			&\left\|
			u_3(t,x)
			-\Gh(t,\xh)\int_{\mathbb{R}^2}u_{0,3}(\yh,x_3)d\yh
			+\nablah \Gh(t,\xh)\cdot \int_{\mathbb{R}^2}\yh u_{0,3}(\yh,x_3)d\yh
			\right.\\
			&\left. \qquad \qquad \qquad \qquad
			-\nablah \Gh(t,\xh) \cdot \int_0^{\infty}\int_{\mathbb{R}^2}(u_3\uh)(\tau,\yh,x_3)d\yh d\tau
			\right\|_{L^p_x}
			=0.
		\end{split}
	\end{align}
\end{thm}
\begin{rem}\label{rem:main_2}
	\begin{itemize}
		\item [(1)]
		The regularity assumption $s\geqslant 9$ is necessary only for the proof of (\ref{main_2:asymp_uh}).
		It is possible to prove (\ref{main_2:asymp_u3}) and (\ref{main_2:asymp_u3_second}) under the weaker assumption $s\geqslant 5$, which is the same regularity condition as in Theorem \ref{thm:main_1}.
		\item [(2)]
		By $t^{-\frac{3}{2}(1- \frac{1}{p})-\frac{1}{2p}}=t^{-(1- \frac{1}{p})- \frac{1}{2}}$, (\ref{main_2:asymp_u3})
		and (\ref{main_2:asymp_u3_second}), we see that the second leading term of $u_3(t)$ decays as the $2$D heat kernel with the horizontal gradient $\nablah e^{t\Deltah}$,
		while the leading term of $u_3(t)$ decays like the $3$D heat kernel $e^{t\Delta}$.
		\item [(3)]
		The estimate (\ref{main_2:asymp_uh}) with $p=1$ holds if we additionally assume $\nabla^{\alpha}u_0 \in \Lh^{\infty}\Lv^1(\mathbb{R}^3)$ ($|\alpha|\leqslant 1$).
		See Remarks \ref{rem:Duha_ exp} and \ref{rem:Lh^infLv1-est}.
		\item [(4)]
		By Proposition \ref{prop:integral_equation} and Lemma \ref{lemm:Duha_est_1} below, the nonlinear effect in the asymptotic expansion (\ref{main_2:asymp_u3_second})
		appears only from the the first Duhamel term $\Nvv_1[u]$ of $u_3(t)$.
		It follows from Lemma \ref{lemm:Duha_est_1} below that
		$L^1$ decay rates of all nonlinear terms $\Nvv_m[u](t)$ ($m=1,2,3$) of $u_3(t)$ are equals except for $\log t$.
		Therefore, in order to accurate the approximation (\ref{main_2:asymp_u3_second}) with $p=1$,
		we need not only information for $\Nvv_1[u](t)$ but information for all $\Nvv_m[u](t)$ ($m=1,2,3$).
	\end{itemize}
\end{rem}
As a corollary of (\ref{main_2:asymp_u3_second}), we see that the limit (\ref{main_1:asymp_u3}) fails if $p=\infty$.
\begin{cor}\label{cor}
	There exists a suitable initial data $u_0\in \mathscr{S}(\mathbb{R}^3)$ such that the corresponding solution $u$ satisfies
	\begin{align*}
		\liminf_{t\to \infty}
		t^{\frac{3}{2}}
		&\left\|
		u_3(t,x)
		-\Gh(t,\xh)\int_{\mathbb{R}^2}u_{0,3}(\yh,x_3)d\yh
		\right\|_{L^{\infty}_x}>0.
	\end{align*}
\end{cor}
This paper is organized as follows.
In Section \ref{s2}, we give a decomposition of the integral equation corresponding to (\ref{eq:ANS_1}), which is the key ingredient of our analysis.
In Section \ref{s3}, we prepare linear estimates.
Nonlinear decay estimates are established in Section \ref{s4}.
In Section \ref{s5}, we present the proof of Theorem \ref{thm:main_1}.
In Section \ref{s6}, we establish the additional estiamtes for the solution.
Finally, in Section \ref{s7}, we prove Theorem \ref{thm:main_2} and Corollary \ref{cor}.
\subsection*{Notation}
At the end of this section, we summarize notation used in this paper.
For given $3$D vector $a=(a_1,a_2,a_3)\in \mathbb{R}^3$,
we denote by $a_{\h}=(a_1,a_2)$ the horizontal components of $a$
and $a_3$ is called the vertical component or the third component of $a$.
The operator $\Deltah:=\partial_1^2+\partial_2^2$ denotes the horizontal Laplacian
and
$\nablah=(\partial_1,\partial_2)$ represents the horizontal gradient.
For $s \in \mathbb{N}$, we define a function space $X^s(\mathbb{R}^3)$ by
\begin{align*}
	X^s(\mathbb{R}^3)
	&:=H^s(\mathbb{R}^3)\cap L^1(\mathbb{R}^2_{\xh};(W^{1,1}\cap W^{1,\infty})(\mathbb{R}_{x_3})).
\end{align*}
Let $\Gh(t,\xh)$ be the $2$D Gaussian:
\begin{align*}
	\Gh(t,\xh):=(4\pi t)^{-1}e^{-\frac{|\xh|^2}{4t}}, \qquad (t, \xh)=(t,x_1,x_2)\in (0,\infty)\times \mathbb{R}^2.
\end{align*}
For $1\leqslant p,q\leqslant \infty$, we define the anisotropic Lebesgue space by
\begin{align*}
	\Lh^p\Lv^q(\mathbb{R}^3):=L^p(\mathbb{R}^2_{\xh};L^q(\mathbb{R}_{x_3})).
\end{align*}
For $d\in \mathbb{N}$, let $\mathscr{F}_{\mathbb{R}^d}[f]$ and $\mathscr{F}^{-1}_{\mathbb{R}^d}[f]$
be the Fourier transform and the inverse Fourier transform of a function $f$ on $\mathbb{R}^d$.
We denote by $C$ the constant, which may differ in each line.
In particular, $C=C(a_1,...,a_n)$ means that $C$ depends only on $a_1,...,a_n$.

\section{Integral Equations}\label{s2}
In this section, we consider the integral equation corresponding to (\ref{eq:ANS_1}):
\begin{align}\label{Int-1}
	u(t)=e^{t\Deltah}u_0-\int_0^te^{(t-\tau)\Deltah}\mathbb{P}\nabla\cdot(u\otimes u)(\tau)d\tau,
\end{align}
where $\mathbb{P}=(\delta_{kl}+R_kR_l)_{1\leqslant k,l\leqslant 3}$ denotes the Helmholtz projection and $\{R_k\}_{k=1}^3$ represents the $3$D Riesz transform.
The unboundedness of the operator $\mathbb{P}$ on $L^1(\mathbb{R}^3)$ and $L^{\infty}(\mathbb{R}^3)$
prevents us from calculating $\|u(t)\|_{L^1}$ and $\|u(t)\|_{L^{\infty}}$ via the integral equation (\ref{Int-1}).
To overcome this, we follow the idea of \cite{FM} and decompose the nonlinear term as follows.
The $j$-th component of $e^{(t-\tau)\Deltah}\mathbb{P}\nabla\cdot(u\otimes u)(\tau)$ is given by
\begin{align}\label{j-th nonlinear}
	\begin{split}
		&\left(e^{(t-\tau)\Deltah}\mathbb{P}\nabla\cdot(u\otimes u)(\tau)\right)_j\\
		&\quad
		=\sum_{k=1}^3
		e^{(t-\tau)\Deltah}\partial_k(u_ku_j)(\tau)
		+
		\sum_{k,l=1}^3\partial_j\partial_k\partial_l \mathscr{F}_{\mathbb{R}^3}^{-1}\left[|\xi|^{-2}e^{-(t-\tau)|\xih|^2}\right]*(u_ku_l)(\tau)\\
		&\quad
		=\sum_{k=1}^3
		e^{(t-\tau)\Deltah}\partial_k(u_ku_j)(\tau)
		+
		\sum_{k,l=1}^2\partial_j\partial_k\partial_lK(t-\tau)*(u_ku_l)(\tau)\\
		&\qquad+
		2\sum_{k=1}^2\partial_j\partial_k\partial_3K(t-\tau)*(u_ku_3)(\tau)
		+
		\partial_j\partial_3\partial_3K(t-\tau)*(u_3(\tau)^2)\\
	\end{split}
\end{align}
where
$K(t,x):=\mathscr{F}_{\mathbb{R}^3}^{-1}\left[|\xi|^{-2}e^{-t|\xih|^2}\right](x)$.
Let us derive the explicit formula for the function $K(t,x)$.
\begin{lemm}\label{lemma:K}
	The following formula holds:
	\begin{align*}
		K(t,x)
		&=
		\mathscr{F}_{\mathbb{R}^2}^{-1}\left[\frac{1}{2|\xih|}e^{-t|\xih|^2}e^{-|\xih||x_3|}\right](\xh)\\
		&=
		\int_0^{\infty}\frac{e^{-\frac{|\xh|^2}{4(t+s)}}}{4\pi(t+s)}\frac{e^{-\frac{x_3^2}{4s}}}{(4\pi s)^{\frac{1}{2}}}ds\\
		&=
		\int_0^{\infty}\Gh(t+s,\xh)\Gv(s,x_3)ds,
	\end{align*}
	where $\Gv(s,x_3)$ is the $1$D Gaussian.
\end{lemm}
\begin{proof}
	Using the formula
	\begin{align*}
		\int_{\mathbb{R}}\frac{e^{ix_3\xi_3}}{\xi_3^2+|\xih|^2}d\xi_3
		=
		\frac{\pi}{|\xih|}e^{-|\xih||x_3|},
	\end{align*}
	we have
	\begin{align*}
		K(t,x)
		&=
		\frac{1}{(2\pi)^3}
		\int_{\mathbb{R}^2}
		e^{i\xh \cdot \xih}
		e^{-t|\xih|^2}
		\int_{\mathbb{R}}\frac{e^{ix_3\xi_3}}{\xi_3^2+|\xih|^2}
		d\xi_3
		d\xih\\
		&=
		\frac{1}{2(2\pi)^2}
		\int_{\mathbb{R}^2}
		e^{i\xh \cdot \xih}
		\frac{1}{|\xih|}
		e^{-t|\xih|^2}
		e^{-|\xih||x_3|}
		d\xih\\
		&=
		\mathscr{F}_{\mathbb{R}^2}^{-1}\left[\frac{1}{2|\xih|}e^{-t|\xih|^2}e^{-|\xih||x_3|}\right](\xh).
	\end{align*}
	On the other hand, we have
	\begin{align*}
		K(t,x)
		&=
		\frac{1}{(2\pi)^3}
		\int_{\mathbb{R}^2}
		e^{i\xh \cdot \xih}
		e^{-t|\xih|^2}
		\int_{\mathbb{R}}\frac{e^{ix_3\xi_3}}{\xi_3^2+|\xih|^2}
		d\xi_3
		d\xih\\
		&=
		\frac{1}{(2\pi)^3}
		\int_{\mathbb{R}^2}
		e^{i\xh \cdot \xih}
		e^{-t|\xih|^2}
		\int_{\mathbb{R}}e^{ix_3\xi_3}\int_0^{\infty}e^{-s(\xi_3^2+|\xih|^2)}
		ds
		d\xi_3
		d\xih\\
		&=
		\int_0^{\infty}
		\frac{1}{(2\pi)^2}
		\int_{\mathbb{R}^2}
		e^{i\xh \cdot \xih}
		e^{-(t+s)|\xih|^2}
		d\xih
		\frac{1}{2\pi}
		\int_{\mathbb{R}}
		e^{ix_3\xi_3}e^{-s\xi_3^2}
		d\xi_3
		ds
		\\
		&=
		\int_0^{\infty}
		\mathscr{F}_{\mathbb{R}^2}^{-1}\left[e^{-(t+s)|\xih|^2}\right](\xh)
		\mathscr{F}_{\mathbb{R}}^{-1}\left[e^{-s|\xi_3|^2}\right](x_3)
		ds\\
		&=
		\int_0^{\infty}
		\frac{e^{-\frac{|\xh|^2}{4(t+s)}}}{4\pi(t+s)}
		\frac{e^{-\frac{x_3^2}{4s}}}{(4\pi s)^{\frac{1}{2}}}
		ds.
	\end{align*}
	This completes the proof.
\end{proof}
By Lemma \ref{lemma:K}, we obtain
\begin{align}\label{del3K}
	\begin{split}
		\partial_3K(t,x)
		&=
		\mathscr{F}_{\mathbb{R}^2}^{-1}\left[\frac{1}{2|\xih|}e^{-t|\xih|^2}\partial_{x_3}(e^{-|\xih||x_3|})\right](\xh)\\
		&=
		-{\rm sgn}(x_3)\mathscr{F}_{\mathbb{R}^2}^{-1}\left[|\xih|\frac{1}{2|\xih|}e^{-t|\xih|^2}e^{-|\xih||x_3|}\right](\xh)\\
		&=
		-{\rm sgn}(x_3)(-\Deltah)^{\frac{1}{2}}K(t,x)
	\end{split}
\end{align}
and
\begin{equation*}
	\partial_3\partial_3K(t,x)
	=
	-\delta(x_3)\mathscr{F}_{\mathbb{R}^2}^{-1}\left[e^{-t|\xih|^2}e^{-|\xih||x_3|}\right](\xh)
	-\Deltah K(t,x),
\end{equation*}
which implies
\begin{align}\label{del33K}
	\partial_3\partial_3K(t)*f
	=-e^{t\Deltah}f-\Deltah K(t)*f.
\end{align}
Here, $\delta(x_3)$ is the Dirac distribution on $\mathbb{R}_{x_3}$.
It follows from (\ref{del3K}) and (\ref{del33K}) that
\begin{align}\label{del333K}
	\begin{split}
		\partial_3\partial_3\partial_3K(t)*f
		&=-e^{t\Deltah}\partial_3f-\Deltah \partial_3K(t)*f\\
		&=-e^{t\Deltah}\partial_3f-\left({\rm sgn}(x_3')(-\Deltah)^{\frac{3}{2}}K(t,x')\right)*f.
	\end{split}
\end{align}
Hence, by (\ref{j-th nonlinear}), (\ref{del3K}), (\ref{del33K}) and (\ref{del333K}), we see that
\begin{align*}
	\left(e^{(t-\tau)\Deltah}\mathbb{P}\nabla\cdot(u\otimes u)(\tau)\right)_j
	&
	=
	\partial_3e^{(t-\tau)\Deltah}(u_3u_j)(\tau)
	+
	\sum_{k=1}^2
	\partial_ke^{(t-\tau)\Deltah}(u_ku_j)(\tau)\\
	&\quad
	+
	\sum_{k,l=1}^2\partial_j\partial_k\partial_lK(t-\tau)*(u_ku_l)(\tau)\\
	&\quad
	-
	2\sum_{k=1}^2\left({\rm sgn}(x_3')\partial_j\partial_k(-\Deltah)^{\frac{1}{2}}K(t-\tau,x')\right)*(u_ku_3)(\tau)\\
	&\quad
	-
	\partial_je^{(t-\tau)\Deltah}(u_3(\tau)^2)
	-
	\partial_j\Deltah K(t-\tau)*(u_3(\tau)^2)
\end{align*}
for $j=1,2$ and
\begin{align*}
	\left(e^{(t-\tau)\Deltah}\mathbb{P}\nabla\cdot(u\otimes u)(\tau)\right)_3
	&
	=
	-
	\sum_{k=1}^2\partial_ke^{(t-\tau)\Deltah}(u_ku_3)(\tau)\\
	&\quad
	-
	\sum_{k,l=1}^2\left({\rm sgn}(x_3')\partial_k\partial_l(-\Deltah)^{\frac{1}{2}}K(t-\tau,x')\right)*(u_ku_l)(\tau)\\
	&\quad
	-
	2\sum_{k=1}^2\partial_k\Deltah K(t-\tau)*(u_ku_3)(\tau)\\
	&\quad
	-
	\left({\rm sgn}(x_3')(-\Deltah)^{\frac{3}{2}}K(t-\tau,x')\right)*(u_3(\tau)^2).
\end{align*}
Therefore, we obtain the following proposition:
\begin{prop}\label{prop:integral_equation}
	Let $u$ be a solution to (\ref{Int-1}).
	Then, the following formula holds:
	\begin{align}\label{eq:Int-2}
		\begin{cases}
			\uh(t)
			=
			e^{t\Deltah}u_{0,\h}
			+\displaystyle\sum_{m=1}^5 \Nhh_m[u](t),\\
			u_3(t)
			=
			e^{t\Deltah}u_{0,3}
			+\displaystyle\sum_{m=1}^3\Nvv_m[u](t),
		\end{cases}
	\end{align}
	where
	\begin{align*}
		\Nhh_1[u](t)&:=-\int_0^te^{(t-\tau)\Deltah}\partial_{3}(u_3\uh)(\tau)d\tau,\\
		\Nhh_2[u](t)&:=-\int_0^te^{(t-\tau)\Deltah}\nablah \cdot (\uh\otimes\uh)(\tau)d\tau,\\
		\Nhh_3[u](t)&:=\int_0^t\nablah e^{(t-\tau)\Deltah}(u_3(\tau)^2)d\tau,\\
		\Nhh_4[u](t)&:=-\sum_{k,l=1}^2\int_0^t\nablah\partial_{k}\partial_{l}K(t-\tau)*(u_ku_l)(\tau)d\tau,\\
		\Nhh_5[u](t)&:=2\sum_{k=1}^2\int_0^t\nablah \partial_k(-\Deltah)^{\frac{1}{2}}\widetilde{K}(t-\tau)*(u_3u_k)(\tau)d\tau+\int_0^t\nablah \Deltah K(t-\tau)*(u_3(\tau)^2)d\tau
	\end{align*}
	and
	\begin{align*}
		\Nvv_1[u](t)&:=\int_0^te^{(t-\tau)\Deltah}\nablah\cdot(u_3\uh)(\tau)d\tau,\\
		\Nvv_2[u](t)&:=\sum_{k,l=1}^2\int_0^t(-\Deltah)^{\frac{1}{2}}\partial_{k}\partial_{l}\widetilde{K}(t-\tau)*(u_ku_l)(\tau)d\tau,\\
		\Nvv_3[u](t)&:=2\sum_{k=1}^2\int_0^t\partial_k\Deltah K(t-\tau)*(u_3u_k)(\tau)d\tau+\int_0^t(-\Deltah)^{\frac{3}{2}}\widetilde{K}(t-\tau)*(u_3(\tau)^2)d\tau.
	\end{align*}
	Here, $K(t,x)$ and $\widetilde{K}(t,x)$ are the functions with the following representaions:
	\begin{align*}
		&K(t,x)=\int_0^{\infty}\frac{e^{-\frac{|\xh|^2}{4(t+s)}}}{4\pi(t+s)}\frac{e^{-\frac{x_3^2}{4s}}}{(4\pi s)^{\frac{1}{2}}}ds
		=\int_0^{\infty}\Gh(t+s,\xh)\Gv(s,x_3)ds,\\
		&\widetilde{K}(t,x)={\rm sgn}(x_3)K(t,x).
	\end{align*}
\end{prop}
\begin{rem}\label{rem:integral_equation}
	\begin{itemize}
		\item [(1)]
		Compared with (\ref{Int-1}), we see that (\ref{eq:Int-2}) possesses no singular integral operator.
		Thus, (\ref{eq:Int-2}) enables us to calculate $L^1$ and $L^{\infty}$ norm of $u(t)$.
		\item [(2)]
		The above decompositions for the nonlinear terms are classified with respect to the $L^p$-decay rates.
		See Lemma \ref{lemm:Duha_est_1} below.
	\end{itemize}
\end{rem}
\section{Linear Analysis}\label{s3}
In this section, we prepare some linear estimates.
We start with the classical preoperties of heat kernels.
\begin{lemm}\label{lemm:Linear_est_1}
	\begin{itemize}
		\item [(1)]
		For each $\alpha=(\alphah,\alphav)\in (\mathbb{N}\cup\{0\})^2 \times (\mathbb{N}\cup\{0\})$ and $m\in \mathbb{N}\cup\{0\}$,
		there exists a positive constant $C=C(\alpha,m)$ such that
		\begin{align*}
			\||\xh|^m\nablah^{\alphah}\Gh(t,\xh)\|_{L^p(\mathbb{R}^2_{\xh})}\leqslant Ct^{-(1- \frac{1}{p})-\frac{|\alphah|}{2}+\frac{m}{2}}
		\end{align*}
		for all $1\leqslant p\leqslant \infty$ and $t>0$.
		In particular, we have
		\begin{equation*}
			\|\nabla^{\alpha}e^{t\Deltah}f\|_{\Lh^{p_2}\Lv^q}
			\leqslant
			Ct^{-(\frac{1}{p_1}- \frac{1}{p_2})-\frac{|\alphah|}{2}}
			\|\partial_3^{\alpha_3}f\|_{\Lh^{p_1}\Lv^q}
		\end{equation*}
		for all $t>0$, $1\leqslant p_1\leqslant p_2\leqslant \infty$, $1\leqslant q\leqslant \infty$ and all functions $f$ satisfying $\partial_3^{\alpha_3}f\in \Lh^{p_1}\Lv^q(\mathbb{R}^3)$.
		Moreover, it holds
		\begin{align*}
			\||\xh|\nabla^{\alpha}e^{t\Deltah}g\|_{\Lh^{p}\Lv^q}
			&\leqslant
			Ct^{-(\frac{1}{p_1} - \frac{1}{p})-\frac{|\alphah|}{2}+\frac{1}{2}}\|\partial_3^{\alpha_3}g\|_{\Lh^{p_1}\Lv^q}\\
			&\quad+
			Ct^{-(\frac{1}{p_2} - \frac{1}{p})-\frac{|\alphah|}{2}}\||\xh|\partial_3^{\alpha_3}g(x)\|_{\Lh^{p_2}\Lv^q}
		\end{align*}
		for all $t>0$, $1\leqslant p_1,p_2\leqslant p\leqslant \infty$, $1\leqslant q\leqslant \infty$
		and all functions $g$ satisfying
		$\partial_3^{\alpha_3}g\in \Lh^{p_1}\Lv^q(\mathbb{R}^3)$
		and
		$|\xh|\partial_3^{\alpha_3}g(x)\in \Lh^{p_2}\Lv^q(\mathbb{R}^3)$.
		\item [(2)]
		Let $1\leqslant p, q\leqslant \infty$ and $m=0,1$.
		Then, for any function $f$ satisfying $|\xh|^mf(x)\in L^1(\mathbb{R}^2_{\xh};(L^1\cap L^{\infty})(\mathbb{R}_{x_3}))$,
		it holds
		\begin{align*}
			\lim_{t\to \infty}t^{(1-\frac{1}{p})+\frac{m}{2}}&
			\left\|
				e^{t\Deltah}f(x)
				-
				\sum_{|\alphah|\leqslant m}\nablah^{\alphah}\Gh(t,\xh)\int_{\mathbb{R}^2}(-\yh)^{\alphah}f(\yh,x_3)d\yh
			\right\|_{\Lh^p\Lv^q}=0.
		\end{align*}
		\item [(3)]
		There exists a positive constant $C$ such that
		\begin{align*}
			\left\|
				e^{t\Deltah}f(x)
				-
				\Gh(t,\xh)\int_{\mathbb{R}^2}f(\yh,x_3)d\yh
			\right\|_{\Lh^p\Lv^q}
			\leqslant
			Ct^{-(1-\frac{1}{p})-\frac{1}{2}}\||\xh|f(x)\|_{\Lh^1\Lv^q}
		\end{align*}
		for all $1\leqslant p, q\leqslant \infty$ and all functions $f$ satisfying $|\xh|f(x)\in \Lh^1\Lv^q(\mathbb{R}^3)$.
	\end{itemize}
\end{lemm}
We omit the proof of Lemma \ref{lemm:Linear_est_1} since it is quite standard.
Next, we forcus on the enhanced dissiaption for $e^{t\Deltah}u_{0,3}$.
\begin{lemm}\label{lemm:Linear_est_2}
	\begin{itemize}
		\item [(1)]
		There exists an absolute positive constant $C$ such that
		\begin{align}\label{enhaced decay}
			\|e^{t\Deltah}u_{0,3}\|_{\Lh^p\Lv^q}
			\leqslant
			Ct^{-(1- \frac{1}{p})-\frac{1}{2}(1- \frac{1}{q})}\|u_0\|_{L^1}
		\end{align}
		for all $1\leqslant p,q\leqslant \infty$, $t>0$ and $u_0=(u_{0,1},u_{0,2},u_{0,3})\in L^1(\mathbb{R}^3)$ satisfying $\nabla \cdot u_0=0$.

		Moreover, it holds
		\begin{align}\label{W-enhaced decay}
			\||\xh|e^{t\Deltah}u_{0,3}(x)\|_{\Lh^1\Lv^{\infty}}
			\leqslant
			C\|(1+|\xh|)u_0(x)\|_{L^1_x}
		\end{align}
		for all $t>0$ and $u_0=(u_{0,1},u_{0,2},u_{0,3})\in L^1(\mathbb{R}^3)$ satisfying $\nabla \cdot u_0=0$ and $|\xh|u_0\in L^1(\mathbb{R}^3)$.
		\item [(2)]
		There exists an absolute positive constant $C$ such that
		for $1\leqslant p\leqslant \infty$, $m=0,1$
		and any function $u_0=(u_{0,1},u_{0,2},u_{0,3})$
		satisfying $\nabla \cdot u_0=0$
		and $|\xh|^mu_0(x)\in L^1(\mathbb{R}^2_{\xh};(L^1\cap L^{\infty})(\mathbb{R}_{x_3}))$,
		there exists
		a nonnegative function $\mathcal{R}_{p,m}(t)$ such that
		$\mathcal{R}_{p,m}(t)\to 0$ as $t\to \infty$
		and
		\begin{align}\label{enhanced_expansion}
			\begin{split}
				&\left\|
				e^{t \Deltah}u_{0,3}(x)-\sum_{|\alphah|\leqslant m}\nablah^{\alphah}\Gh(t,\xh)\int_{\mathbb{R}^2}(-\yh)^{\alphah}u_{0,3}(\yh,x_3) d\yh
				\right\|_{L^p}\\
				&\qquad
				\leqslant
				Ct^{-\frac{3}{2}(1- \frac{1}{p})-\frac{m}{2}}\mathcal{R}_{p,m}(t)^{\frac{1}{p}}\||\xh|^m u_0(x)\|_{L^1}^{1- \frac{1}{p}}
			\end{split}
		\end{align}
		for all $t>0$.
	\end{itemize}
\end{lemm}
\begin{proof}
	We first prove (\ref{enhaced decay}).
	Using $\partial_3u_{0,3}=-\nablah\cdot u_{0,\h}$, we obtain that
	\begin{align*}
		\|e^{t\Deltah}u_{0,3}\|_{\Lh^p\Lv^q}
		&\leqslant
		\|e^{t\Deltah}u_{0,3}\|_{\Lh^p\Lv^1}^{\frac{1}{q}}
		\| e^{t\Deltah}u_{0,3}\|_{\Lh^p\Lv^{\infty}}^{1-\frac{1}{q}}\\
		&\leqslant
		C
		\left(
		t^{-(1- \frac{1}{p})}
		\| u_{0,3}\|_{L^1}
		\right)^{\frac{1}{q}}
		\|e^{t\Deltah}\partial_3 u_{0,3}\|_{\Lh^p\Lv^1}^{1-\frac{1}{q}}\\
		&=
		C
		t^{-\frac{1}{q}(1- \frac{1}{p})}
		\|u_{0,3}\|_{L^1}^{\frac{1}{q}}
		\|e^{t\Deltah}\nablah \cdot u_{0,\h}\|_{\Lh^p\Lv^1}^{1-\frac{1}{q}}\\
		&\leqslant
		C
		t^{-\frac{1}{q}(1- \frac{1}{p})}
		\|u_{0,3}\|_{L^1}^{\frac{1}{q}}
		\left(
		t^{-(1-\frac{1}{p})-\frac{1}{2}}
		\|u_{0,\h}\|_{L^1}
		\right)^{1-\frac{1}{q}}\\
		&\leqslant
		Ct^{-(1-\frac{1}{p})- \frac{1}{2}(1- \frac{1}{q})}\|u_0\|_{L^1}.
	\end{align*}
	For the proof of (\ref{W-enhaced decay}), we see by Lemma \ref{lemm:Linear_est_1} that
	\begin{align*}
		\||\xh|e^{t\Deltah}u_{0,3}(x)\|_{\Lh^1\Lv^{\infty}}
		&\leqslant
		\||\xh|e^{t\Deltah}\partial_3u_{0,3}(x)\|_{L^1}\\
		&=
		\||\xh|e^{t\Deltah}\nablah \cdot u_{0,\h}(x)\|_{L^1}\\
		&\leqslant
		C\|(1+|\xh|)u_0(x)\|_{L^1}.
	\end{align*}

	Next, we show (\ref{enhanced_expansion}).
	Let
	\begin{equation*}
		F_m(t,x)
		:=
		e^{t \Deltah}u_{0,3}(x)-\sum_{|\alphah|\leqslant m}\nablah^{\alphah}\Gh(t,\xh)\int_{\mathbb{R}^2}(-\yh)^{\alphah}u_{0,3}(\yh,x_3)d\yh.
	\end{equation*}
	Then, we see that
	\begin{align}\label{Fm_Lp}
		\|F_m(t)\|_{L^p}
		\leqslant
		\| F_m(t)\|_{\Lh^p\Lv^1}^{\frac{1}{p}}
		\| F_m(t)\|_{\Lh^p\Lv^{\infty}}^{1-\frac{1}{p}}
		\leqslant
		\| F_m(t)\|_{\Lh^p\Lv^1}^{\frac{1}{p}}
		\| \partial_3 F_m(t)\|_{\Lh^p\Lv^1}^{1-\frac{1}{p}}.
	\end{align}
	Here, by Lemma \ref{lemm:Linear_est_1} (2) , we see that
	\begin{align}\label{Rm}
		\| F_m(t)\|_{\Lh^p\Lv^1}
		=t^{-(1-\frac{1}{p})- \frac{m}{2}}\mathcal{R}_{p,m}(t),\qquad
		\lim_{t\to \infty}\mathcal{R}_{p,m}(t)=0.
	\end{align}
	For the case (\ref{enhanced_expansion}) $m=0$, it follows from the divergence free condition on $u_0$ and integration by parts that
	\begin{align}
		&\partial_3\int_{\mathbb{R}^2}u_{0,3}(\yh,x_3)d\yh
		=
		-\int_{\mathbb{R}^2}\nablah\cdot u_{0,\h}(\yh,x_3)d\yh
		=0,\label{zero1}\\
		&\partial_3e^{t\Deltah}u_{0,3}(x)
		=
		-\int_{\mathbb{R}^2}\nablah\Gh(t,\xh-\yh)\cdot u_{0,\h}(\yh)d\yh.\label{zero2}
	\end{align}
	Thus, we have
	\begin{align*}
		 \partial_3F_0(t,x)
		=-\int_{\mathbb{R}^2}(\nablah\Gh)(t,\xh-\yh)\cdot u_{0,\h}(\yh,x_3)d\yh,
	\end{align*}
	which implies
	\begin{align}\label{del3F_0}
		\begin{split}
			\| \partial_3F_0(t,x)\|_{\Lh^p\Lv^1}
			&\leqslant
			\|\nablah\Gh(t)\|_{L^p(\mathbb{R}^2)}\| u_{0,\h}\|_{L^1}\\
			&\leqslant
			Ct^{-(1-\frac{1}{p})-\frac{1}{2}}\|u_0\|_{L^1}.
		\end{split}
	\end{align}
	Hence, we obtain by (\ref{Fm_Lp}), (\ref{Rm}) and (\ref{del3F_0}) that
	\begin{equation*}
		\|F_0(t)\|_{L^p}
		\leqslant
		Ct^{-\frac{3}{2}(1- \frac{1}{p})}\mathcal{R}_{p,0}(t)^{\frac{1}{p}}\|u_0\|_{L^1}^{1- \frac{1}{p}}.
	\end{equation*}
	This gives (\ref{enhanced_expansion}) with $m=0$.
	For the case $m=1$, it holds by the integration by parts and the divergence free condition that
	\begin{align}\label{zero3}
		\int_{\mathbb{R}^2}y_k\partial_3u_{0,3}(\yh,x_3) d\yh
		=-\int_{\mathbb{R}^2}y_k\nablah \cdot u_{0,\h}(\yh,x_3) d\yh
		=\int_{\mathbb{R}^2}u_{0,k}(\yh,x_3)d\yh.
	\end{align}
	Using (\ref{zero1}), (\ref{zero2}) and (\ref{zero3}), we have
	\begin{align*}
		 \partial_3F_1(t,x)
		&=
		-\int_{\mathbb{R}^2}\{(\nablah\Gh)(t,\xh-\yh)-(\nablah \Gh)(t,\xh)\}\cdot u_{0,\h}(\yh,x_3)d\yh\\
		&=
		\int_{\mathbb{R}^2}\int_0^1(\nablah^2\Gh)(t,\xh-\theta\yh)\yh\cdot u_{0.\h}(\yh,x_3)d\theta d\yh.
	\end{align*}
	Taking $\Lh^p\Lv^1$-norm, we see that
	\begin{align*}
		\| \partial_3 F_1(t)\|_{\Lh^p\Lv^1}
		\leqslant
		Ct^{-(1-\frac{1}{p})-1}\||\xh|u_0(x)\|_{L^1}.
	\end{align*}
	Hence, we obtain
	\begin{align*}
		\|F_1(t)\|_{L^p}
		\leqslant
		Ct^{-\frac{3}{2}(1- \frac{1}{p})-\frac{1}{2}}\mathcal{R}_{p,1}(t)^{\frac{1}{p}}\||\xh|u_0(x)\|_{L^1}^{1- \frac{1}{p}},
	\end{align*}
	which completes the proof.
\end{proof}

Next, we recall the function
\begin{align*}
	K(t,x)
	&=
	\int_0^{\infty}\frac{e^{-\frac{|\xh|^2}{4(t+s)}}}{4\pi(t+s)}\frac{e^{-\frac{x_3^2}{4s}}}{(4\pi s)^{\frac{1}{2}}}ds\\
	&=
	\int_0^{\infty}\Gh(t+s,\xh)\Gv(s,x_3)ds,
\end{align*}
which is defined in Section \ref{s2} (see Lemma \ref{lemma:K}).
In the following lemma, we state the decay rate for the derivative of $K(t)$.
\begin{lemm}\label{lemm:Linear_est_3}
	Let $1\leqslant p,q\leqslant \infty$, $(\beta,\gamma)\in (\mathbb{N}\cup\{0\})^2\times (\mathbb{N}\cup\{0\})$ and $m\in \mathbb{N}\cup\{0\}$
	satisfy
	\begin{align*}
		|\beta|+\gamma > \frac{2}{p}+\frac{1}{q}-1+m.
	\end{align*}
	Then,
	there exists a positive constant
	$C=C(p,q,\beta,\gamma, m)$ such that
	\begin{align*}
		&\left\||\xh|^m\nablah^{\beta}(-\Deltah)^{\frac{\gamma}{2}}K(t,x)\right\|_{\Lh^p\Lv^{q}}
		\leqslant
		Ct^{-(1- \frac{1}{p})-\frac{1}{2}(1- \frac{1}{q})-\frac{|\beta|+\gamma-2}{2}+\frac{m}{2}}
	\end{align*}
	for all $t>0$.

	Moreover, if
	$1\leqslant p_1,p_2\leqslant p$ and $1\leqslant q_1,q_2\leqslant q$ satisfy
	\begin{align*}
		\left(\frac{1}{p_1}-\frac{1}{p}\right)+\frac{1}{2}\left(\frac{1}{q_1} - \frac{1}{q}\right)+\frac{|\beta|+\gamma-3}{2}>0,\qquad
		\left(\frac{1}{p_2}-\frac{1}{p}\right)+\frac{1}{2}\left(\frac{1}{q_2} - \frac{1}{q}\right)+\frac{|\beta|+\gamma-2}{2}>0,
	\end{align*}
	then
	there exists a positive constant
	$C=C(p,p_1,p_2,q,q_1,q_2,\beta,\gamma)$ such that
	\begin{align*}
		\||\xh|\nablah^{\beta}(-\Deltah)^{\frac{\gamma}{2}}K(t)*f(x)\|_{\Lh^p\Lv^q}
		&\leqslant
		Ct^{-(\frac{1}{p_1}-\frac{1}{p})-\frac{1}{2}(\frac{1}{q_1} - \frac{1}{q})-\frac{|\beta|+\gamma-3}{2}}\|f\|_{\Lh^{p_1}\Lv^{q_1}}\\
		&\quad+
		Ct^{-(\frac{1}{p_2}-\frac{1}{p})-\frac{1}{2}(\frac{1}{q_2} - \frac{1}{q})-\frac{|\beta|+\gamma-2}{2}}\||\xh|f(x)\|_{\Lh^{p_2}\Lv^{q_2}}
	\end{align*}
	for all $t>0$
	and all functions $f$ satisfying
	$\partial_3^{\alpha_3}f\in \Lh^{p_1}\Lv^{q_1}(\mathbb{R}^3)$
	and
	$|\xh|\partial_3^{\alpha_3}f(x)\in \Lh^{p_2}\Lv^{q_2}(\mathbb{R}^3)$.
\end{lemm}
\begin{rem}\label{rem:Linear_est_3}
	Lemma \ref{lemm:Linear_est_3} implies that
	the operator $\nablah^{\beta}(-\Deltah)^{\frac{\gamma}{2}}K(t)*$ decays as the $3$D heat kernel with the horizontal gradient of $|\beta|+\gamma-2$-th order.
\end{rem}
\begin{proof}[Proof of Lemma \ref{lemm:Linear_est_3}]
	From Lemma \ref{lemma:K}, it follows that
	\begin{equation*}\label{beta-gamma-K}
		|\xh|^m\nablah^{\beta}(-\Deltah)^{\frac{\gamma}{2}}K(t,x)
		=\int_0^{\infty}\Gv(s,x_3)|\xh|^m\nablah^{\beta}(-\Deltah)^{\frac{\gamma}{2}}\Gh(t+s,\xh)ds.
	\end{equation*}
	Therefore, we have
	\begin{align*}
		\left\|
		|\xh|^m\nablah^{\beta}(-\Deltah)^{\frac{\gamma}{2}}K(t,x)
		\right\|_{\Lh^p\Lv^{q}}
		&\leqslant
		\int_0^{\infty}
		\|\Gv(s)\|_{L^{q}(\mathbb{R})}
		\left\||\xh|^m\nablah^{\beta}(-\Deltah)^{\frac{\gamma}{2}}\Gh(t+s)\right\|_{L^p(\mathbb{R}^2)}
		ds\\
		&\leqslant
		C\int_0^{\infty}
		s^{-\frac{1}{2}(1- \frac{1}{q})}
		(t+s)^{-(1- \frac{1}{p})-\frac{|\beta|+\gamma}{2}+\frac{m}{2}}
		ds\\
		&=
		Ct^{-(1- \frac{1}{p})-\frac{1}{2}(1- \frac{1}{q})-\frac{|\beta|+\gamma-2}{2}+\frac{m}{2}},
	\end{align*}
	which proves the first estimate.
	For the second estimate,
	we see that
	\begin{align*}
		|\xh|\left|\nablah^{\beta}(-\Deltah)^{\frac{\gamma}{2}}K(t)*f(x)\right|
		&\leqslant
		\int_{\mathbb{R}^3}|\xh-\yh|\left|\nablah^{\beta}(-\Deltah)^{\frac{\gamma}{2}}K(t,x-y)\right||f(y)|dy\\
		&\quad+
		\int_{\mathbb{R}^3}\left|\nablah^{\beta}(-\Deltah)^{\frac{\gamma}{2}}K(t,x-y)\right||\yh||f(y)|dy.
	\end{align*}
	Hence, taking $\Lh^p\Lv^q$-norm and applying the Hausdorff-Young inequality, we complete the proof.
\end{proof}
\section{Nonlinear Analysis}\label{s4}
In this section, we establish decay estimates and asymptotic expansions for the Duhamel terms related to the vector fields $u=(u_1(t,x),u_2(t,x),u_3(t,x))$,
which satisfies some of the following assumptions for some $s\in \mathbb{N}$, $0<T\leqslant \infty$ and $A,B\geqslant 0$:
\begin{itemize}
	\item [{\bf (A1)}]
	$u\in C([0,\infty);X^s(\mathbb{R}^3))$,
	$\nabla\cdot u=0$ and
	\begin{align*}
		\|u(t)\|_{H^s}\leqslant A,\qquad
		\|u(t)\|_{L^1(\mathbb{R}^2_{\xh};(W^{1,1}\cap W^{1,\infty})(\mathbb{R}_{x_3}))}\leqslant A(1+t)
	\end{align*}
	for all $t> 0$.\\
	\item [{\bf (A2)}]
	For $1\leqslant p\leqslant \infty$, there hold
	\begin{align*}
		&\|\nabla^{\alpha}\uh(t)\|_{L^p}\leqslant CAt^{-(1-\frac{1}{p})- \frac{|\alphah|}{2}}, \\
		&\|\nablah^{\alphah}u_3(t)\|_{L^p}\leqslant CAt^{-\frac{3}{2}(1-\frac{1}{p})- \frac{|\alphah|}{2}}
	\end{align*}
	for all $0<t<T$ and $\alpha\in (\mathbb{N}\cup\{0\})^3$ with $|\alpha|\leqslant 1$.\\
	\item [{\bf (A3)}]
	$\|\nabla^{\alpha}u(t)\|_{\Lh^{\infty}\Lv^1}\leqslant At^{-1- \frac{|\alphah|}{2}}$
	for all $0<t<T$ and $\alpha\in (\mathbb{N}\cup\{0\})^3$ with $|\alpha|\leqslant 1$.\\
	\item [{\bf (A4)}]
	There hold
	\begin{align*}
		\||\xh|\uh(t,x)\|_{\Lh^1\Lv^{\infty}}\leqslant B(1+t)^{\frac{1}{2}},\qquad
		\||\xh|u_3(t,x)\|_{\Lh^1\Lv^{\infty}}\leqslant B
	\end{align*}
	for all $t> 0$.
\end{itemize}
\begin{rem}\label{rem:assumption}
	It is easy to check that if $u$ satisfies {\bf (A1)} and {\bf (A2)} for some $s\in \mathbb{N}$ with $s\geqslant 3$, $0<T\leqslant \infty$ and $A\geqslant 0$,
	then there exists an absolute positive constant $C$ such that
	\begin{align*}
		&\|\uh(t)\|_{L^p},\|\partial_3\uh(t)\|_{L^p}
		\leqslant CA(1+t)^{-(1-\frac{1}{p})},\\
		&\|\nablah \uh(t)\|_{L^p}
		\leqslant
		\begin{cases}
			CA (1+t)^{-(1-\frac{1}{p})-\frac{1}{2}} & (2\leqslant p\leqslant \infty)\\
			CA t^{-(1-\frac{1}{p})-\frac{1}{2}} & (1\leqslant p< 2)
		\end{cases},\\
		&\|u_3(t)\|_{L^p}
		\leqslant CA(1+t)^{-\frac{3}{2}(1-\frac{1}{p})},\\
		&\|\partial_3u_3(t)\|_{L^p}
		\leqslant CA(1+t)^{-(1-\frac{1}{p})-\frac{1}{2}},\\
		&\|\nablah u_3(t)\|_{L^p}
		\leqslant
		\begin{cases}
			CA (1+t)^{-\frac{3}{2}(1-\frac{1}{p})-\frac{1}{2}} & (2\leqslant p\leqslant \infty)\\
			CA t^{-\frac{3}{2}(1-\frac{1}{p})-\frac{1}{2}} & (1\leqslant p< 2)
		\end{cases},\\
		&\|\uh(t)\|_{\Lh^1\Lv^{\infty}}
		\leqslant CA,\\
		&\|u_3(t)\|_{\Lh^1\Lv^{\infty}}
		\leqslant
		CA(1+\tau)^{-\frac{1}{2}}
	\end{align*}
	for all $1\leqslant p\leqslant \infty$ and $0<t<T$.
	In the following of this paper, we often use this fact.
\end{rem}
\subsection*{Decay Estimates for the Duhamel Terms}
We first forcus on the decay rates for the Duhamel terms.
\begin{lemm}\label{lemm:Duha_est_1}
	Let $u$ satisfy {\bf (A1)} and {\bf (A2)} for some $s\in \mathbb{N}$ with $s \geqslant 5$, $0<T\leqslant \infty$ and $A\geqslant 0$.
	Then, there exists an absolute positive constant $C$ such that
	\begin{align*}
		&
		\|\nabla^{\alpha}\Nhh_1[u](t)\|_{L^p}\leqslant CA^2(1+t)^{-(1- \frac{1}{p})-\frac{|\alphah|}{2}},\\
		&
		\|\nabla^{\alpha}\Nhh_2[u](t)\|_{L^p}\leqslant CA^2(1+t)^{-(1- \frac{1}{p})- \frac{1+|\alphah|}{2}}\log (2+t),\\
		&
		\|\nabla^{\alpha}\Nhh_3[u](t)\|_{L^p}\leqslant CA^2(1+t)^{-(1- \frac{1}{p})- \frac{1+|\alphah|}{2}},\\
		&
		\|\nabla^{\alpha}\Nhh_4[u](t)\|_{L^p}\leqslant CA^2(1+t)^{-\frac{9}{8}(1- \frac{1}{p})- \frac{1+|\alphah|}{2}}\log (2+t),\\
		&
		\|\nabla^{\alpha}\Nhh_5[u](t)\|_{L^p}\leqslant CA^2(1+t)^{-\frac{9}{8}(1- \frac{1}{p})- \frac{1+|\alphah|}{2}}
	\end{align*}
	and
	\begin{align*}
		&
		\|\nabla^{\alpha}\Nvv_1[u](t)\|_{L^p}\leqslant CA^2(1+t)^{-(1- \frac{1}{p})- \frac{1+|\alphah|}{2}}
													=CA^2(1+t)^{-\frac{3}{2}(1- \frac{1}{p})- \frac{1}{2p}-\frac{|\alphah|}{2}},\\
		&
		\|\nabla^{\alpha}\Nvv_2[u](t)\|_{L^p}\leqslant CA^2(1+t)^{-\frac{9}{8}(1- \frac{1}{p})- \frac{1+|\alphah|}{2}}\log(2+t),\\
		&
		\|\nabla^{\alpha}\Nvv_3[u](t)\|_{L^p}\leqslant CA^2(1+t)^{-\frac{9}{8}(1- \frac{1}{p})- \frac{1+|\alphah|}{2}}
	\end{align*}
	for all $1\leqslant p\leqslant \infty$, $0<t<T$  and $\alpha\in (\mathbb{N}\cup\{0\})^3$ with $|\alpha|\leqslant 1$.
\end{lemm}
\begin{proof}
	It suffices to prove
	\begin{align}
			&\left\|
			\nabla^{\alpha}
			\int_0^t e^{(t-\tau)\Deltah}\partial_3(u_3\uh)(\tau)d\tau
			\right\|_{L^p}
			\leqslant
			CA^2(1+t)^{-(1- \frac{1}{p})-\frac{|\alphah|}{2}},\label{Duha_est_1-1}\\
			&\left\|
			\nabla^{\alpha}
			\int_0^t e^{(t-\tau)\Deltah}\nablah(u_ku_l)(\tau)d\tau
			\right\|_{L^p}
			\leqslant
			\begin{cases}
				CA^2(1+t)^{-(1- \frac{1}{p})-\frac{1+|\alphah|}{2}}\log(2+t) & (k=1,2)\\
				CA^2(1+t)^{-(1- \frac{1}{p})-\frac{1+|\alphah|}{2}} & (k=3)
			\end{cases},\label{Duha_est_2-1}\\
			&\left\|
			\nabla^{\alpha}
			\int_0^t K^{(m)}_{\beta,\gamma}(t-\tau)*(u_ku_l)(\tau)d\tau
			\right\|_{L^p}
			\nonumber\\
			&\qquad \qquad
			\leqslant
			\begin{cases}
				CA^2(1+t)^{-\frac{9}{8}(1- \frac{1}{p})-\frac{1+|\alphah|}{2}}\log(2+t) & (k=1,2)\\
				CA^2(1+t)^{-\frac{9}{8}(1- \frac{1}{p})-\frac{1+|\alphah|}{2}} & (k=3)
			\end{cases},\label{Duha_est_3-1}
	\end{align}
	for $l=1,2,3$, $m=1,2$ and $0<t<T$,
	where
	\begin{align}\label{K^m}
		K^{(1)}_{\beta,\gamma}(t,x):=\nablah^{\beta}(-\Deltah)^{\frac{\gamma}{2}}K(t,x),\qquad
		K^{(2)}_{\beta,\gamma}(t,x):={\rm sgn}(x_3)\nablah^{\beta}(-\Deltah)^{\frac{\gamma}{2}}K(t,x)
	\end{align}
	for $(\beta,\gamma)\in (\mathbb{N}\cup \{0\})^2\times (\mathbb{N}\cup \{0\})$ satisfying $|\beta|+\gamma=3$.
	The interpolation yields that it is enough to prove (\ref{Duha_est_1-1})-(\ref{Duha_est_3-1}) only for the case $p=1$ and $p=\infty$.

	First, we show (\ref{Duha_est_1-1}).
	For the case $p=1$,
	using the estimates
	\begin{align*}
		\|\partial_3(u_3\uh)(\tau)\|_{L^1}
		&\leqslant
			\|\partial_3u_3(\tau)\|_{L^{\infty}}
			\|\uh(\tau)\|_{L^1}
			+
			\|u_3(\tau)\|_{L^{\infty}}
			\|\partial_3\uh(\tau)\|_{L^1}\\
		&\leqslant
			CA^2(1+\tau)^{-\frac{3}{2}}
	\end{align*}
	and
	\begin{align}\label{del3^2(u_3uh)_L^1}
		\begin{split}
			\|\partial_3^2(u_3\uh)(\tau)\|_{L^1}
			&\leqslant
				\|\partial_3^2u_3(\tau)\|_{L^2}
				\|\uh(\tau)\|_{L^2}
				+
				2
				\|\partial_3u_3(\tau)\|_{L^2}
				\|\partial_3\uh(\tau)\|_{L^2}\\
				&\qquad \qquad \qquad
				+
				\|u_3(\tau)\|_{L^2}
				\|\partial_3^2\uh(\tau)\|_{L^2}\\\
			&\leqslant
			C
			\left(
				\|\partial_3u_3(\tau)\|_{L^2}^{\frac{2}{3}}
				\|\partial_3^4u_3(\tau)\|_{L^2}^{\frac{1}{3}}
				\|\uh(\tau)\|_{L^2}
				+
				\|\partial_3u_3(\tau)\|_{L^2}
				\|\partial_3\uh(\tau)\|_{L^2}\right.\\
				&\qquad \qquad \qquad
				\left.+
				\|u_3(\tau)\|_{L^2}
				\|\partial_3\uh(\tau)\|_{L^2}^{\frac{2}{3}}
				\|\partial_3^4\uh(\tau)\|_{L^2}^{\frac{1}{3}}
			\right)\\
			&\leqslant
				CA^2(1+\tau)^{-\frac{13}{12}},
		\end{split}
	\end{align}
	we see that
	\begin{align*}
		\left\|
		\nabla^{\alpha}
		\int_0^t e^{(t-\tau)\Deltah}\partial_3(u_3\uh)(\tau)d\tau
		\right\|_{L^1}
		&\leqslant
		C
		\int_0^t
		\|\nablah^{\alphah}\Gh(t-\tau)\|_{L^1(\mathbb{R}^2)}
		\|\partial_3^{\alpha_3+1}(u_3\uh)(\tau)\|_{L^1}d\tau\\
		&\leqslant
		CA^2
		\int_0^t
		(t-\tau)^{- \frac{|\alphah|}{2}}
		(1+\tau)^{-\frac{13}{12}}d\tau\\
		&\leqslant
		CA^2(1+t)^{- \frac{|\alphah|}{2}}.
	\end{align*}
	For the case $p=\infty$ with $\alpha_3=0$,
	it is easy to see that
	\begin{align*}
		\|\partial_3(u_3\uh)(\tau)\|_{\Lh^1\Lv^{\infty}}
		&\leqslant
			\|\partial_3u_3(\tau)\|_{L^{\infty}}
			\|\uh(\tau)\|_{\Lh^1\Lv^{\infty}}
			+
			\|u_3(\tau)\|_{\Lh^1\Lv^{\infty}}
			\|\partial_3\uh(\tau)\|_{L^{\infty}}\\
		&\leqslant
			CA^2(1+\tau)^{-\frac{3}{2}}
	\end{align*}
	and
	\begin{align*}
		\|\partial_3(u_3\uh)(\tau)\|_{L^{\infty}}
		&
		\leqslant
		\|\partial_3u_3(\tau)\|_{L^{\infty}}
		\|\uh(\tau)\|_{L^{\infty}}
		+
		\|u_3(\tau)\|_{L^{\infty}}
		\|\partial_3\uh(\tau)\|_{L^{\infty}}\\
		&
		\leqslant
		CA^2(1+\tau)^{-\frac{5}{2}},
	\end{align*}
	which yield
	\begin{align*}
		&\left\|
		\nablah^{\alphah}
		\int_0^t e^{(t-\tau)\Deltah}\partial_3(u_3\uh)(\tau)d\tau
		\right\|_{L^{\infty}}\\
		&\quad\leqslant
		C
		\int_0^{\frac{t}{2}}
		\|\nablah^{\alphah}\Gh(t-\tau)\|_{L^{\infty}(\mathbb{R}^2)}
		\|\partial_3(u_3\uh)(\tau)\|_{\Lh^1\Lv^{\infty}}d\tau
		+
		C
		\int_{\frac{t}{2}}^t
		\|\nablah^{\alphah}\Gh(t-\tau)\|_{L^1(\mathbb{R}^2)}
		\|\partial_3(u_3\uh)(\tau)\|_{L^{\infty}}d\tau\\
		&\quad\leqslant
		CA^2
		\int_0^{\frac{t}{2}}
		(t-\tau)^{-1- \frac{|\alphah|}{2}}
		(1+\tau)^{-\frac{3}{2}}d\tau
		+
		CA^2
		\int_{\frac{t}{2}}^t
		(t-\tau)^{- \frac{|\alphah|}{2}}
		(1+\tau)^{-\frac{5}{2}}d\tau\\
		&\quad\leqslant
		CA^2t^{-1-\frac{|\alphah|}{2}}
	\end{align*}
	for $t\geqslant 1$ and
	\begin{align*}
		\left\|
		\nablah^{\alphah}
		\int_0^t e^{(t-\tau)\Deltah}\partial_3(u_3\uh)(\tau)d\tau
		\right\|_{L^{\infty}}
		&\leqslant
		C
		\int_0^t
		\|\nablah^{\alphah}\Gh(t-\tau)\|_{L^1(\mathbb{R}^2)}
		\|\partial_3(u_3\uh)(\tau)\|_{L^{\infty}}d\tau\\
		&\leqslant
		CA^2
		\int_{0}^t
		(t-\tau)^{- \frac{|\alphah|}{2}}
		(1+\tau)^{- \frac{5}{2}}
		d\tau\\
		&\leqslant
		CA^2
	\end{align*}
	for $0< t \leqslant 1$.
	For the case $p=\infty$ with $\alpha_3=1$,
	we have by the Gagliardo-Nirenberg interpolation inequality that
	\begin{align*}
		\|\partial_3^2(u_3\uh)(\tau)\|_{\Lh^1\Lv^{\infty}}
		&\leqslant
			\|\partial_3^2u_3(\tau)\|_{\Lh^2\Lv^{\infty}}
			\|\uh(\tau)\|_{\Lh^2\Lv^{\infty}}
			+
			2
			\|\partial_3u_3(\tau)\|_{\Lh^2\Lv^{\infty}}
			\|\partial_3\uh(\tau)\|_{\Lh^2\Lv^{\infty}}\\
			&\qquad \qquad \qquad
			+
			\|u_3(\tau)\|_{\Lh^2\Lv^{\infty}}
			\|\partial_3^2\uh(\tau)\|_{\Lh^2\Lv^{\infty}}\\\
		&\leqslant
		C
		\left(
			\|\partial_3u_3(\tau)\|_{L^2}^{\frac{5}{8}}
			\|\partial_3^5u_3(\tau)\|_{L^2}^{\frac{3}{8}}
			\|\uh(\tau)\|_{L^2}^{\frac{1}{2}}
			\|\partial_3\uh(\tau)\|_{L^2}^{\frac{1}{2}}\right.\\
			&\qquad \qquad \qquad
			+
			\|\partial_3u_3(\tau)\|_{L^2}^{\frac{5}{6}}
			\|\partial_3^4u_3(\tau)\|_{L^2}^{\frac{1}{6}}
			\|\partial_3\uh(\tau)\|_{L^2}^{\frac{5}{6}}
			\|\partial_3^4\uh(\tau)\|_{L^2}^{\frac{1}{6}}\\
			&\qquad \qquad \qquad
			\left.+
			\|u_3(\tau)\|_{L^2}^{\frac{1}{2}}
			\|\partial_3u_3(\tau)\|_{L^2}^{\frac{1}{2}}
			\|\partial_3\uh(\tau)\|_{L^2}^{\frac{5}{8}}
			\|\partial_3^5\uh(\tau)\|_{L^2}^{\frac{3}{8}}
		\right)\\
		&\leqslant
			CA^2(1+\tau)^{-\frac{9}{8}}
	\end{align*}
	and also
	\begin{align*}
		\|\partial_3^2(u_3\uh)(\tau)\|_{\Lh^2\Lv^{\infty}}
		&\leqslant
			\|\partial_3^2u_3(\tau)\|_{\Lh^2\Lv^{\infty}}
			\|\uh(\tau)\|_{L^{\infty}}
			+
			2
			\|\partial_3u_3(\tau)\|_{L^{\infty}}
			\|\partial_3\uh(\tau)\|_{\Lh^2\Lv^{\infty}}\\
			&\qquad \qquad \qquad
			+
			\|u_3(\tau)\|_{L^{\infty}}
			\|\partial_3^2\uh(\tau)\|_{\Lh^2\Lv^{\infty}}\\
		&\leqslant
		C
		\left(
			\|\partial_3u_3(\tau)\|_{L^2}^{\frac{5}{8}}
			\|\partial_3^5u_3(\tau)\|_{L^2}^{\frac{3}{8}}
			\|\uh(\tau)\|_{L^{\infty}}\right.\\
			&\qquad \qquad \qquad
			\left.+
			\|\partial_3u_3(\tau)\|_{L^{\infty}}
			\|\partial_3\uh(\tau)\|_{L^2}^{\frac{5}{6}}
			\|\partial_3^4\uh(\tau)\|_{L^2}^{\frac{1}{6}}\right.\\
			&\qquad \qquad \qquad
			\left.+
			\|u_3(\tau)\|_{L^{\infty}}
			\|\partial_3\uh(\tau)\|_{L^2}^{\frac{5}{8}}
			\|\partial_3^5\uh(\tau)\|_{L^2}^{\frac{3}{8}}
		\right)\\
		&\leqslant
			CA^2(1+\tau)^{-\frac{13}{8}}.
	\end{align*}
	Combining these estimates, we obtain
	\begin{align*}
		&\left\|
		\partial_3
		\int_0^t e^{(t-\tau)\Deltah}\partial_3(u_3\uh)(\tau)d\tau
		\right\|_{L^{\infty}}\\
		&\quad\leqslant
		C
		\int_0^{\frac{t}{2}}
		\|\Gh(t-\tau)\|_{L^{\infty}(\mathbb{R}^2)}
		\|\partial_3^{2}(u_3\uh)(\tau)\|_{\Lh^1\Lv^{\infty}}d\tau\\
		&\qquad+
		C
		\int_{\frac{t}{2}}^t
		\|\Gh(t-\tau)\|_{L^2(\mathbb{R}^2)}
		\|\partial_3^2(u_3\uh)(\tau)\|_{\Lh^2\Lv^{\infty}}d\tau\\
		&\quad\leqslant
		CA^2
		\int_0^{\frac{t}{2}}
		(t-\tau)^{-1}
		(1+\tau)^{-\frac{9}{8}}d\tau
		+
		CA^2
		\int_{\frac{t}{2}}^t
		(t-\tau)^{-\frac{1}{2}}
		(1+\tau)^{-\frac{13}{8}}d\tau\\
		&\quad\leqslant
		CA^2(1+t)^{-1}
	\end{align*}
	for $t>0$.

	We next show (\ref{Duha_est_2-1}).
	For the case $p=1$, we easily see that
	\begin{align}\label{pat3(u_ku_l)_L^1}
		\begin{split}
			\|\partial_3^{\alpha_3}(u_ku_l)(\tau)\|_{L^1}
			&\leqslant
			2\|\partial_3^{\alpha_3}u(\tau)\|_{L^1}\|u(\tau)\|_{L^{\infty}}
			\leqslant
			CA^2(1+\tau)^{-1} \qquad {\rm for\ }k=1,2,\\
			\|\partial_3^{\alpha_3}(u_3u_l)(\tau)\|_{L^1}
			&\leqslant
			\begin{cases}
				\|u_3(\tau)\|_{L^{\infty}}\|u(\tau)\|_{L^1} & (\alpha_3=0)\\
				\|\partial_3u_3(\tau)\|_{L^{\infty}}\|u(\tau)\|_{L^1}
				+\|u_3(\tau)\|_{L^{\infty}}\|\partial_3u(\tau)\|_{L^1} & (\alpha_3=1)
			\end{cases}\\
			&\leqslant
			CA^2(1+\tau)^{- \frac{3}{2}}
		\end{split}
	\end{align}
	and also have
	\begin{align*}
		\|\nabla^{\alpha}(u_ku_l)\|_{L^1}
		\leqslant
		2\|\nabla^{\alpha}u(\tau)\|_{L^{\infty}}\|u(\tau)\|_{L^1}
		\leqslant
		CA^2(1+\tau)^{-1-\frac{|\alphah|}{2}}.
	\end{align*}
	Let $\sigma_k=1$ (if $k=1,2$) and $\sigma_3=3/2$.
	Then, we have
	\begin{align*}
		&\left\|
		\nabla^{\alpha}
		\int_0^t e^{(t-\tau)\Deltah}\nablah(u_ku_l)(\tau)d\tau
		\right\|_{L^1}\\
		&\quad
		\leqslant
		\int_0^{\frac{t}{2}}
		\|\nablah^{\alphah}\nablah\Gh(t-\tau)\|_{L^1(\mathbb{R}^2)}\|\partial_3^{\alpha_3}(u_ku_l)(\tau)\|_{L^1}d\tau
		+
		\int_{\frac{t}{2}}^t
		\|\nablah\Gh(t-\tau)\|_{L^1(\mathbb{R}^2)}\|\nabla^{\alpha}(u_ku_l)(\tau)\|_{L^1}d\tau\\
		&\quad
		\leqslant
		CA^2
		\int_0^{\frac{t}{2}}
		(t-\tau)^{-\frac{1+|\alphah|}{2}}(1+\tau)^{-\sigma_k}d\tau
		+
		CA^2
		\int_{\frac{t}{2}}^t
		(t-\tau)^{-\frac{1}{2}}(1+\tau)^{-1- \frac{|\alphah|}{2}}d\tau\\
		&\quad\leqslant
		\begin{cases}
			CA^2(1+t)^{-\frac{1+|\alphah|}{2}}\log(2+t) & (k=1,2)\\
			CA^2(1+t)^{-\frac{1+|\alphah|}{2}} & (k=3)
		\end{cases}.
	\end{align*}
	For the case $p=\infty$, we have
	\begin{align*}
		\|\partial_3^{\alpha_3}(u_ku_l)(\tau)\|_{\Lh^1\Lv^{\infty}}
		&\leqslant
		C\|\partial_3^{\alpha_3}u(\tau)\|_{L^{\infty}}\|u(\tau)\|_{\Lh^1\Lv^{\infty}}
		\leqslant
		CA^2(1+\tau)^{-1} \qquad {\rm for\ }k=1,2,\\
		\|\partial_3^{\alpha_3}(u_3u_l)(\tau)\|_{\Lh^1\Lv^{\infty}}
		&\leqslant
		\begin{cases}
			\|u_3(\tau)\|_{L^{\infty}}\|\uh(\tau)\|_{\Lh^1\Lv^{\infty}} & (\alpha_3=0)\\
			\|\partial_3u_3(\tau)\|_{L^{\infty}}\|\uh(\tau)\|_{\Lh^1\Lv^{\infty}}
			+
			\|u_3(\tau)\|_{\Lh^1\Lv^{\infty}}\|\partial_3\uh(\tau)\|_{L^{\infty}} & (\alpha_3=1)
		\end{cases}\\
		&\leqslant
		CA^2(1+\tau)^{-\frac{3}{2}}
	\end{align*}
	and
	\begin{align*}
		\|\nabla^{\alpha}(u_ku_l)(\tau)\|_{L^{\infty}}
		\leqslant
		C\|\nabla^{\alpha}u(\tau)\|_{L^{\infty}}\|u(\tau)\|_{L^{\infty}}
		\leqslant
		CA^2(1+\tau)^{-2- \frac{|\alphah|}{2}}.
	\end{align*}
	Then, we obtain
	\begin{align*}
		&\left\|
		\nabla^{\alpha}
		\int_0^t e^{(t-\tau)\Deltah}\nablah(u_ku_l)(\tau)d\tau
		\right\|_{L^{\infty}}\\
		&\quad
		\leqslant
		\int_0^{\frac{t}{2}}
		\|\nablah^{\alphah}\nablah\Gh(t-\tau)\|_{L^{\infty}(\mathbb{R}^2)}
		\|\partial_3^{\alpha_3}(u_ku_l)(\tau)\|_{\Lh^1\Lv^{\infty}}d\tau
		+
		\int_{\frac{t}{2}}^t
		\|\nablah\Gh(t-\tau)\|_{L^1(\mathbb{R}^2)}
		\|\nabla^{\alpha}(u_ku_l)(\tau)\|_{L^{\infty}}d\tau\\
		&\quad
		\leqslant
		CA^2
		\int_0^{\frac{t}{2}}
		(t-\tau)^{-1-\frac{1+|\alphah|}{2}}(1+\tau)^{-\sigma_k}d\tau
		+
		CA^2
		\int_{\frac{t}{2}}^t
		(t-\tau)^{-\frac{1}{2}}(1+\tau)^{-2- \frac{|\alphah|}{2}}d\tau\\
		&\quad\leqslant
		\begin{cases}
			CA^2t^{-1-\frac{1+|\alphah|}{2}}\log(2+t) & (k=1,2)\\
			CA^2t^{-1-\frac{1+|\alphah|}{2}} & (k=3)
		\end{cases}
	\end{align*}
	for $t\geqslant 1$
	and
	\begin{align*}
		\left\|
		\nabla^{\alpha}
		\int_0^t e^{(t-\tau)\Deltah}\nablah(u_ku_l)(\tau)d\tau
		\right\|_{L^{\infty}}
		&
		\leqslant
		C\int_0^t
		(t-\tau)^{-\frac{1}{2}}\|\nabla^{\alpha}(u_ku_l)(\tau)\|_{L^{\infty}}d\tau\\
		&
		\leqslant
		CA^2
		\int_0^t
		(t-\tau)^{-\frac{1}{2}}(1+\tau)^{-2- \frac{|\alphah|}{2}}d\tau\\
		&\leqslant
		CA^2
	\end{align*}
	for $0< t \leqslant 1$.

	We finally prove (\ref{Duha_est_3-1}).
	For the case $p=1$,
	it follows from Lemma \ref{lemm:Linear_est_3} and the inequalities in the proof of (\ref{Duha_est_2-1}) that
	\begin{align*}
		&\left\|
		\nabla^{\alpha}
		\int_0^t K^{(m)}_{\beta,\gamma}(t-\tau)*(u_ku_l)(\tau)d\tau
		\right\|_{L^1}\\
		&\quad
		\leqslant
		C\int_0^{\frac{t}{2}}
		\|\nablah^{\alphah}K^{(m)}_{\beta,\gamma}(t-\tau)\|_{L^1}\|\partial_3^{\alpha_3}(u_ku_l)(\tau)\|_{L^1}d\tau
		+
		C\int_{\frac{t}{2}}^t
		\|K^{(m)}_{\beta,\gamma}(t-\tau)\|_{L^1}\|\nabla^{\alpha}(u_ku_l)(\tau)\|_{L^1}d\tau\\
		&\quad
		\leqslant
		CA^2
		\int_0^{\frac{t}{2}}
		(t-\tau)^{-\frac{1+|\alphah|}{2}}(1+\tau)^{-\sigma_k}d\tau
		+
		CA^2
		\int_{\frac{t}{2}}^t
		(t-\tau)^{-\frac{1}{2}}(1+\tau)^{-1- \frac{|\alphah|}{2}}d\tau\\
		&\quad\leqslant
		\begin{cases}
			CA^2(1+t)^{-\frac{1+|\alphah|}{2}}\log(2+t) & (k=1,2)\\
			CA^2(1+t)^{-\frac{1+|\alphah|}{2}} & (k=3)
		\end{cases}.
	\end{align*}
	For the case $p=\infty$, we see that
	\begin{align*}
		&\left\|
		\nabla^{\alpha}
		\int_0^t K^{(m)}_{\beta,\gamma}(t-\tau)*(u_ku_l)(\tau)d\tau
		\right\|_{L^{\infty}}\\
		&\quad
		\leqslant
		C\int_0^{\frac{t}{2}}
		\|\nablah^{\alphah}K^{(m)}_{\beta,\gamma}(t-\tau)\|_{L^{\infty}}\|\partial_3^{\alpha_3}(u_ku_l)(\tau)\|_{L^1}d\tau
		+
		C\int_{\frac{t}{2}}^t
		\|K^{(m)}_{\beta,\gamma}(t-\tau)\|_{L^{\frac{4}{3}}}\|\nabla^{\alpha}(u_ku_l)(\tau)\|_{L^4}d\tau\\
		&\quad
		\leqslant
		C
		\int_0^{\frac{t}{2}}
		(t-\tau)^{-\frac{3}{2}-\frac{1+|\alphah|}{2}}\|\partial_3^{\alpha_3}(u_ku_l)(\tau)\|_{L^1}d\tau
		+
		C
		\int_{\frac{t}{2}}^t
		(t-\tau)^{-\frac{7}{8}}\|\nabla^{\alpha}u(\tau)\|_{L^{\infty}}\|u(\tau)\|_{L^4}d\tau\\
		&\quad
		\leqslant
		CA^2
		\int_0^{\frac{t}{2}}
		(t-\tau)^{-\frac{3}{2}-\frac{1+|\alphah|}{2}}(1+\tau)^{-\sigma_k}d\tau
		+
		CA^2
		\int_{\frac{t}{2}}^t
		(t-\tau)^{-\frac{7}{8}}(1+\tau)^{-\frac{7}{4}- \frac{|\alphah|}{2}}d\tau\\
		&\quad\leqslant
		\begin{cases}
			CA^2t^{-\frac{9}{8}-\frac{1+|\alphah|}{2}}\log(2+t) & (k=1,2)\\
			CA^2t^{-\frac{9}{8}-\frac{1+|\alphah|}{2}} & (k=3)
		\end{cases}
	\end{align*}
	for $t\geqslant 1$  and
	\begin{align*}
		\left\|
		\nabla^{\alpha}
		\int_0^t K^{(m)}_{\beta,\gamma}(t-\tau)*(u_ku_l)(\tau)d\tau
		\right\|_{L^{\infty}}
		&
		\leqslant
		C\int_0^t
		\|K^{(m)}_{\beta,\gamma}(t-\tau)\|_{L^1}\|\nabla^{\alpha}u(\tau)\|_{L^{\infty}}\|u_l(\tau)\|_{L^{\infty}}d\tau\\
		&
		\leqslant
		CA^2
		\int_0^t
		(t-\tau)^{-\frac{1}{2}}(1+\tau)^{-2- \frac{|\alphah|}{2}}d\tau\\
		&\leqslant
		CA^2
	\end{align*}
	for $0< t\leqslant 1$.
	Thus, we complete the proof.
\end{proof}
\begin{lemm}\label{lemm:Duha_est_4}
	Let $u$ satisfy {\bf (A1)}  and {\bf (A2)} for some $s\in \mathbb{N}$ with $s \geqslant 9$, $0<T\leqslant \infty$ and $A\geqslant 0$.
	Then, there exists an absolute positive constant $C$ such that
	\begin{align*}
		\|\nabla^{\alpha}\Nhh_m[u](t)\|_{\Lh^{\infty}\Lv^1}\leqslant CA^2t^{-1 -\frac{|\alphah|}{2}},\qquad
		\|\nabla^{\alpha}\Nvv_n[u](t)\|_{\Lh^{\infty}\Lv^1}\leqslant CA^2t^{-1 -\frac{|\alphah|}{2}}
	\end{align*}
	for all $m=1,2,3,4,5$, $n=1,2,3$, $0<t<T$ and $\alpha\in (\mathbb{N}\cup\{0\})^3$ with $|\alpha|\leqslant 1$.
\end{lemm}
\begin{proof}
	It suffices to show
	\begin{align}
		&\left\|\nabla^{\alpha}\int_0^t e^{(t-\tau)\Deltah}\partial_3(u_3\uh)(\tau)d\tau\right\|_{\Lh^{\infty}\Lv^1}
		\leqslant
		C A^2 t^{-1- \frac{|\alphah|}{2}},\label{Duha_est_4-1}\\
		&\left\|\nabla^{\alpha}\int_0^t e^{(t-\tau)\Deltah}\nablah(u_ku_l)(\tau)d\tau\right\|_{\Lh^{\infty}\Lv^1}
		\leqslant
		C A^2 t^{-1- \frac{|\alphah|}{2}},\label{Duha_est_4-2}\\
		&\left\|\nabla^{\alpha}\int_0^t K^{(m)}_{\beta,\gamma}(t-\tau)*(u_ku_l)(\tau)d\tau\right\|_{\Lh^{\infty}\Lv^1}
		\leqslant
		C A^2 t^{-1- \frac{|\alphah|}{2}}\label{Duha_est_4-3}
	\end{align}
	for $k,l=1,2,3$, $m=1,2$ and $0<t<T$,
	where $K_{\beta,\gamma}^{(m)}$ are defined by (\ref{K^m})
	for $(\beta,\gamma)\in (\mathbb{N}\cup \{0\})^2\times (\mathbb{N}\cup \{0\})$ satisfying $|\beta|+\gamma=3$.

	We first show (\ref{Duha_est_4-1}).
	For the case $\alpha_3=0$, we have
	\begin{align*}
		&\left\|\nablah^{\alphah}\int_0^te^{(t-\tau)\Deltah}\partial_3(u_3\uh)(\tau)d\tau\right\|_{\Lh^{\infty}\Lv^1}\\
		&\quad\leqslant
		\int_0^{\frac{t}{2}}
		\|\nablah^{\alphah}\Gh(t-\tau)\|_{L^{\infty}(\mathbb{R}^2)}
		\|\partial_3(u_3\uh)(\tau)\|_{L^1}d\tau
		+
		\int_{\frac{t}{2}}^t
		\|\nablah^{\alphah}\Gh(t-\tau)\|_{L^1(\mathbb{R}^2)}
		\|\partial_3(u_3\uh)(\tau)\|_{\Lh^{\infty}\Lv^1}d\tau\\
		&\quad\leqslant
		C\int_0^{\frac{t}{2}}
		(t-\tau)^{-1- \frac{|\alphah|}{2}}
		\left(\|\partial_3u_3(\tau)\|_{L^{\infty}}\|\uh(\tau)\|_{L^1}+\|u_3(\tau)\|_{L^{\infty}}\|\partial_3\uh(\tau)\|_{L^1}\right)d\tau\\
		&\qquad+
		C\int_{\frac{t}{2}}^t
		(t-\tau)^{-\frac{|\alphah|}{2}}
		\left(\|\partial_3u_3(\tau)\|_{L^{\infty}}\|\uh(\tau)\|_{\Lh^{\infty}\Lv^1}+\|u_3(\tau)\|_{L^{\infty}}\|\partial_3\uh(\tau)\|_{\Lh^{\infty}\Lv^1}\right)d\tau\\
		&\quad\leqslant
		CA^2\int_0^{\frac{t}{2}}
		(t-\tau)^{-1- \frac{|\alphah|}{2}}(1+\tau)^{-\frac{3}{2}}d\tau
		+
		CA^2\int_{\frac{t}{2}}^t
		(t-\tau)^{-\frac{|\alphah|}{2}}
		(1+\tau)^{-\frac{3}{2}}\tau^{-1}d\tau\\
		&\quad\leqslant
		CA^2t^{-1- \frac{|\alphah|}{2}}.
	\end{align*}
	For the case $\alpha_3=1$,
	it follows from the Gagliardo-Nirenberg interpolation inequality that
	\begin{align}
		&\|\partial_3^2f\|_{L^{\infty}}
		\leqslant
		C\|\partial_3^5f\|_{\Lh^{\infty}\Lv^2}^{\frac{2}{7}}\|\partial_3f\|_{L^{\infty}}^{\frac{5}{7}}
		\leqslant
		C\|f\|_{H^7}^{\frac{2}{7}}\|\partial_3f\|_{L^{\infty}}^{\frac{5}{7}},\label{GW}\\
		&\|\partial_3^2f\|_{\Lh^{\infty}\Lv^2}
		\leqslant
		C\|\partial_3^7f\|_{\Lh^{\infty}\Lv^2}^{\frac{3}{13}}\|\partial_3f\|_{\Lh^{\infty}\Lv^1}^{\frac{10}{13}}
		\leqslant
		C\|f\|_{H^9}^{\frac{3}{13}}\|\partial_3f\|_{\Lh^{\infty}\Lv^1}^{\frac{10}{13}}.\label{GW2}
	\end{align}
	Using (\ref{GW}) and (\ref{GW2}), we obtain
	\begin{align}\label{del3^2(u_3uh)_Lh^infLv1}
		\begin{split}
			\|\partial_3^2(u_3\uh)(\tau)\|_{\Lh^{\infty}\Lv^1}
			&\leqslant
			\|\partial_3^2u_3(\tau)\|_{L^{\infty}}\|\uh(\tau)\|_{\Lh^{\infty}\Lv^1}
			+2\|\partial_3u_3(\tau)\|_{L^{\infty}}\|\partial_3\uh(\tau)\|_{\Lh^{\infty}\Lv^1}\\
			&\quad
			+\|u_3(\tau)\|_{\Lh^{\infty}\Lv^2}\|\partial_3^2\uh(\tau)\|_{\Lh^{\infty}\Lv^2}\\
			&\leqslant
			C\left(
				\|u_3(\tau)\|_{H^7}^{\frac{2}{7}}\|\partial_3u_3(\tau)\|_{L^{\infty}}^{\frac{5}{7}}\|\uh(\tau)\|_{\Lh^{\infty}\Lv^1}\right.\\
				&\qquad \quad
				+
				\|\partial_3u_3(\tau)\|_{L^{\infty}}\|\partial_3\uh(\tau)\|_{\Lh^{\infty}\Lv^1}\\
				&\left.\qquad \quad
				+
				\|u_3(\tau)\|_{L^{\infty}}^{\frac{1}{2}}\|u_3(\tau)\|_{\Lh^{\infty}\Lv^1}^{\frac{1}{2}}
				\|\uh(\tau)\|_{H^9}^{\frac{3}{13}}\|\partial_3\uh(\tau)\|_{\Lh^{\infty}\Lv^1}^{\frac{10}{13}}
			\right)\\
			&\leqslant
			CA^2
			\left\{
			(1+\tau)^{-\frac{3}{4}}\tau^{-\frac{33}{26}}
			+
			(1+\tau)^{-\frac{15}{14}}\tau^{-1}
			\right\}.
		\end{split}
	\end{align}
	By (\ref{del3^2(u_3uh)_L^1}) and (\ref{del3^2(u_3uh)_Lh^infLv1}),
	it holds
	\begin{align*}
		&\left\|\partial_3\int_0^te^{(t-\tau)\Deltah}\partial_3(u_3\uh)(\tau)d\tau\right\|_{\Lh^{\infty}\Lv^1}\\
		&\quad \leqslant
		\int_0^{\frac{t}{2}}
		\|\Gh(t-\tau)\|_{L^{\infty}}\|\partial_3^2(u_3\uh)(\tau)\|_{L^1}d\tau
		+
		\int_{\frac{t}{2}}^t
		\|\Gh(t-\tau)\|_{L^1}\|\partial_3^2(u_3\uh)(\tau)\|_{\Lh^{\infty}\Lv^1}d\tau\\
		&\quad \leqslant
		CA^2\int_0^{\frac{t}{2}}
		(t-\tau)^{-1}(1+\tau)^{-\frac{13}{12}}d\tau
		+
		CA^2\int_{\frac{t}{2}}^t
		(1+\tau)^{-\frac{3}{4}}\tau^{-\frac{33}{26}}
		+
		(1+\tau)^{-\frac{15}{14}}\tau^{-1}
		d\tau\\
		&\quad\leqslant
		CA^2t^{-1},
	\end{align*}
	which completes the proof of (\ref{Duha_est_4-1}).

	Next, we show (\ref{Duha_est_4-2}).
	It is easy to see that
	\begin{align}
		\|\nabla^{\alpha}(u_ku_l)(\tau)\|_{\Lh^{\infty}\Lv^1}
		\leqslant
		2\|u(\tau)\|_{\Lh^{\infty}\Lv^1}\|\nabla^{\alpha}u(\tau)\|_{L^{\infty}}
		\leqslant
		CA^2\tau^{-1}(1+\tau)^{-1- \frac{|\alphah|}{2}}.\label{pat3(u_ku_l)_Lh^infLv1}
	\end{align}
	It follows from (\ref{pat3(u_ku_l)_L^1}) and (\ref{pat3(u_ku_l)_Lh^infLv1}) that
	\begin{align}\label{pat3(u_ku_l)_Lh^2Lv1}
		\begin{split}
			\|\partial_3^{\alpha_3}(u_ku_l)(\tau)\|_{\Lh^2\Lv^1}
			&\leqslant
			\|\partial_3^{\alpha_3}(u_ku_l)(\tau)\|_{L^1}^{\frac{1}{2}}
			\|\partial_3^{\alpha_3}(u_ku_l)(\tau)\|_{\Lh^{\infty}\Lv^1}^{\frac{1}{2}}\\
			&\leqslant
			CA^2\tau^{-\frac{1}{2}}(1+\tau)^{-\frac{3}{2}}.
		\end{split}
	\end{align}
	Hence, by (\ref{pat3(u_ku_l)_Lh^infLv1}) and (\ref{pat3(u_ku_l)_Lh^2Lv1}), we have
	\begin{align*}
		&\left\|
		\nabla^{\alpha}
		\int_0^t e^{(t-\tau)\Deltah}\nablah(u_ku_l)(\tau) d\tau
		\right\|_{\Lh^{\infty}\Lv^1}\\
		&\quad\leqslant
		\int_0^{\frac{t}{2}}
		\|\nabla^{\alphah}\nablah\Gh(t-\tau)\|_{L^2(\mathbb{R}^2)}\|\partial_3^{\alpha_3}(u_ku_l)(\tau)\|_{\Lh^2\Lv^1}d\tau
		+
		\int_{\frac{t}{2}}^t
		\|\nablah \Gh(t-\tau)\|_{L^1(\mathbb{R}^2)}\|\nabla^{\alpha}(u_ku_l)(\tau)\|_{\Lh^{\infty}\Lv^1}d\tau\\
		&\quad\leqslant
		CA^2\int_0^{\frac{t}{2}}
		(t-\tau)^{-1-\frac{|\alphah|}{2}}\tau^{-\frac{1}{2}}(1+\tau)^{-\frac{3}{2}}d\tau
		+
		CA^2\int_{\frac{t}{2}}^t
		(t-\tau)^{-\frac{1}{2}}\tau^{-1}(1+\tau)^{-1-\frac{|\alphah|}{2}}d\tau\\
		&\quad\leqslant
		CA^2t^{-1-\frac{|\alphah|}{2}}.
	\end{align*}
	This completes the proof of (\ref{Duha_est_4-2}).

	Finally, we prove (\ref{Duha_est_4-3}).
	We obtain from (\ref{pat3(u_ku_l)_Lh^infLv1}) and (\ref{pat3(u_ku_l)_Lh^2Lv1}) that
	\begin{align*}
		&\left\|
		\nabla^{\alpha}
		\int_0^t K^{(m)}_{\beta,\gamma}(t-\tau)*(u_ku_l)(\tau) d\tau
		\right\|_{\Lh^{\infty}\Lv^1}\\
		&\quad\leqslant
		\int_0^{\frac{t}{2}}
		\|\nablah^{\alphah}K^{(m)}_{\beta,\gamma}(t-\tau)\|_{\Lh^2\Lv^1}\|\partial_3^{\alpha_3}(u_ku_l)(\tau)\|_{\Lh^2\Lv^1}d\tau
		+
		\int_{\frac{t}{2}}^t
		\|K^{(m)}_{\beta,\gamma}(t-\tau)\|_{L^1}\|\nabla^{\alpha}(u_ku_l)(\tau)\|_{\Lh^{\infty}\Lv^1}d\tau\\
		&\quad\leqslant
		CA^2\int_0^{\frac{t}{2}}
		(t-\tau)^{-1- \frac{|\alphah|}{2}}\tau^{-\frac{1}{2}}(1+\tau)^{-\frac{3}{2}}d\tau
		+
		CA^2\int_{\frac{t}{2}}^t
		(t-\tau)^{-\frac{1}{2}}\tau^{-1}(1+\tau)^{-1- \frac{|\alphah|}{2}}d\tau\\
		&\quad\leqslant
		CA^2t^{-1- \frac{|\alphah|}{2}}.
	\end{align*}
	The proof is completed.
\end{proof}
\subsection*{Asymptotic Expansions of the Duhamel Terms}
Next, we consider the asymptotic expansions of Duhamel terms $\Nhh_1[u]$ and $\Nvv_1[u]$.
\begin{lemm}\label{lemm:Duha_expa_1}
	Let $u$ satisfy {\bf (A1)} and {\bf (A2)} for some $s\in \mathbb{N}$ with $s\geqslant 3$, $A>0$ and $T=\infty$.
	Then, for $1\leqslant p\leqslant \infty$, it holds
	\begin{align}\label{Duha_expa_1-1}
			\begin{split}
				\lim_{t\to \infty}
				t^{1- \frac{1}{p}}
				&\left\|
				\Nhh_1[u](t,x)
				+
				\Gh(t,\xh)
				\int_0^{\infty}
				\int_{\mathbb{R}^2}\partial_3(u_3\uh)(\tau,\yh,x_3)d\yh d\tau
				\right\|_{L^p_x}
				=0.
			\end{split}
	\end{align}
	If, in addition, $u$ satisfies {\bf (A3)} and {\bf (A4)}, then for $1< p\leqslant \infty$ there exists a positive constant $C=C(p)$ such that
	\begin{align}\label{Duha_expa_1-2}
			\begin{split}
				&\left\|
				\Nhh_1[u](t,x)
				+
				\Gh(t,\xh)
				\int_0^{\infty}
				\int_{\mathbb{R}^2}\partial_3(u_3\uh)(\tau,\yh,x_3)d\yh d\tau
				\right\|_{L^p_x}\\
				&\qquad\leqslant
				CA(A+B)
				t^{-(1- \frac{1}{p})- \frac{1}{2}}
				\log(2+t)
			\end{split}
	\end{align}
	for all $t>0$.
\end{lemm}
\begin{proof}
	For the proof of (\ref{Duha_expa_1-1}), we first note that it holds
	\begin{align*}
		&\|\partial_3(u_3\uh)(\tau)\|_{\Lh^1\Lv^p}\\
		&\quad\leqslant
		\|\partial_3u_3(\tau)\|_{L^p}\|\uh(\tau)\|_{\Lh^{p'}\Lv^{\infty}}
		+
		\|u_3(\tau)\|_{\Lh^{p'}\Lv^{\infty}}\|\partial_3\uh(\tau)\|_{L^p}\\
		&\quad\leqslant
		\|\partial_3u_3(\tau)\|_{L^p}\|\uh(\tau)\|_{\Lh^1\Lv^{\infty}}^{1- \frac{1}{p}}\|\uh(\tau)\|_{L^{\infty}}^{\frac{1}{p}}
		+\|u_3(\tau)\|_{\Lh^1\Lv^{\infty}}^{1- \frac{1}{p}}\|u_3(\tau)\|_{L^{\infty}}^{\frac{1}{p}}\|\partial_3\uh(\tau)\|_{L^{p}}\\
		&\quad\leqslant
		CA^2(1+\tau)^{-\frac{3}{2}}
		\in L_{\tau}^1(0,\infty).
	\end{align*}

	Following the idea of \cite[Theorem 4.1]{FM}, let us decompose our target as follows:
	\begin{align*}
		\Nhh_1[u](t,x)
		+
		\Gh(t,\xh)
		\int_0^{\infty}
		\int_{\mathbb{R}^2}\partial_3(u_3\uh)(\tau,\yh,x_3)d\yh d\tau
		=
		-\sum_{m=1}^4I_m(t,x),
	\end{align*}
	where
	\begin{align*}
		I_1(t,x)&=\int_0^{\frac{t}{2}}\int_{\mathbb{R}^2}\left\{\Gh(t-\tau,\xh-\yh)-\Gh(t,\xh-\yh)\right\}\partial_3(u_3\uh)(\tau,\yh,x_3)d\yh d\tau ,\\
		I_2(t,x)&=\int_0^{\frac{t}{2}}\int_{\mathbb{R}^2}\left\{\Gh(t,\xh-\yh)-\Gh(t,\xh)\right\}\partial_3(u_3\uh)(\tau,\yh,x_3)d\yh d\tau ,\\
		I_3(t,x)&=\int_{\frac{t}{2}}^t e^{(t-\tau)\Deltah}\partial_3(u_3\uh)(\tau,x)d\tau,\\
		I_4(t,x)&=-\Gh(t,\xh)\int_{\frac{t}{2}}^{\infty}\int_{\mathbb{R}^2}\partial_3(u_3\uh)(\tau,\yh,x_3)d\yh d\tau.
	\end{align*}
	On the estimate for $I_1(t)$, since
	\begin{align*}
		I_1(t,x)=
		-\int_0^{\frac{t}{2}}
		\int_{\mathbb{R}^2}
		\int_0^1
		\tau(\partial_t\Gh)(t-\theta\tau,\xh-\yh)
		\partial_3(u_3\uh)(\tau,\yh,x_3)d\theta d\yh d\tau,
	\end{align*}
	we have
	\begin{align*}
		\|I_1(t)\|_{L^p}
		&\leqslant
		\int_0^{\frac{t}{2}}
		\int_0^1
		\tau\|(\partial_t\Gh)(t-\theta\tau)\|_{L^p(\mathbb{R}^2)}
		\|\partial_3(u_3\uh)(\tau)\|_{\Lh^1\Lv^p}d\theta d\tau\\
		&\leqslant
		CA^2
		\int_0^{\frac{t}{2}}
		\int_0^1
		\tau(t-\theta\tau)^{-(1- \frac{1}{p})-1}
		(1+\tau)^{-\frac{3}{2}}
		d\theta d\tau\\
		&\leqslant
		CA^2t^{-(1- \frac{1}{p})-1}
		\int_0^{t}
		(1+\tau)^{-\frac{1}{2}}d\tau\\
		&\leqslant
		CA^2t^{-(1- \frac{1}{p})-\frac{1}{2}}.
	\end{align*}
	On the estimate for $I_2(t)$, we see that
	\begin{align*}
		\|I_2(t)\|_{L^p}
		&\leqslant
		t^{-(1- \frac{1}{p})}
		\int_0^{\frac{t}{2}}\int_{\mathbb{R}^2}
		\|\Gh(1,\cdot-t^{-\frac{1}{2}}\yh)-\Gh(1,\cdot)\|_{L^p(\mathbb{R}^2)}
		\|\partial_3(u_3\uh)(\tau,\yh,\cdot)\|_{L^p(\mathbb{R})}d\yh d\tau.
	\end{align*}
	By virture of $\partial_3(u_3\uh)\in L^1(0,\infty;\Lh^1\Lv^p(\mathbb{R}^3))$, the dominated convergence theorem yields that
	$t^{1- \frac{1}{p}}\|I_2(t)\|_{L^p}\to 0$ as $t\to \infty$.

	On the estimate for $I_3(t)$ and $I_4(t)$, we have
	\begin{align*}
		\|I_3(t)\|_{L^p}+\|I_4(t)\|_{L^p}
		&\leqslant
		\int_{\frac{t}{2}}^{t}
		\|\Gh(t-\tau)\|_{L^p(\mathbb{R}^2)}
		\|\partial_3(u_3\uh)(\tau)\|_{\Lh^1\Lv^p}
		d\tau\\
		&\quad+
		\|\Gh(t)\|_{L^p(\mathbb{R}^2)}
		\int_{\frac{t}{2}}^{\infty}
		\|\partial_3(u_3\uh)(\tau)\|_{\Lh^1\Lv^p}
		d\tau\\
		&\leqslant
		CA^2\int_{\frac{t}{2}}^t(t-\tau)^{-(1- \frac{1}{p})}(1+\tau)^{-\frac{3}{2}}d\tau
		+
		CA^2t^{-(1- \frac{1}{p})}
		\int_{\frac{t}{2}}^{\infty}
		\tau^{-\frac{3}{2}}
		d\tau\\
		&\leqslant
		CA^2t^{-(1- \frac{1}{p})-\frac{1}{2}}.
	\end{align*}
	Collecting the estimates for $I_m(t)$ ($m=1,2,3,4$), we obtain (\ref{Duha_expa_1-1}).

	For the proof of (\ref{Duha_expa_1-2}), it suffices to improve the estimate for $I_2(t)$.
	By {\bf (A3)} and $\partial_3 u_3=-\nablah\cdot \uh$, we see that
	\begin{align}\label{I2modfyL1}
		\begin{split}
			\|\partial_3 u_3(\tau)\|_{\Lh^{\infty}\Lv^1}
			&=
			\begin{cases}
				\|\partial_3 u_3(\tau)\|_{\Lh^{\infty}\Lv^1} & (0<\tau<1)\\
				\|\nablah \cdot \uh(\tau)\|_{\Lh^{\infty}\Lv^1} & (\tau \geqslant 1)
			\end{cases}\\
			&\leqslant
			\begin{cases}
				A\tau^{-1} & (0<\tau<1)\\
				A\tau^{-\frac{3}{2}} & (\tau \geqslant 1)
			\end{cases}\\
			&\leqslant
			CA\tau^{-1}(1+\tau)^{-\frac{1}{2}}.
		\end{split}
	\end{align}
	By (\ref{I2modfyL1}), {\bf (A3)} and {\bf (A4)}, we have
	\begin{align*}
		\||\yh|\partial_3(u_3\uh)(\tau)\|_{L^1}
		&\leqslant
		\|\partial_3u_3(\tau)\|_{\Lh^{\infty}\Lv^1}\||\yh|\uh(\tau)\|_{\Lh^1\Lv^{\infty}}
		+\||\yh|u_3(\tau)\|_{\Lh^1\Lv^{\infty}}\|\partial_3\uh(\tau)\|_{\Lh^{\infty}\Lv^1}\\
		&\leqslant
		CAB\tau^{-1}
	\end{align*}
	and
	\begin{align*}
		\||\yh|\partial_3(u_3\uh)(\tau)\|_{\Lh^1\Lv^{\infty}}
		&\leqslant
		\|\partial_3u_3(\tau)\|_{L^{\infty}}\||\yh|\uh(\tau)\|_{\Lh^1\Lv^{\infty}}
		+
		\||\yh|u_3(\tau)\|_{\Lh^1\Lv^{\infty}}\|\uh(\tau)\|_{L^{\infty}}\\
		&\leqslant
		CAB(1+\tau)^{-1}.
	\end{align*}
	Thus, we get
	\begin{align}\label{I2modfyLh1Lvp}
		\begin{split}
			\||\yh|\partial_3(u_3\uh)(\tau)\|_{\Lh^1\Lv^p}
			&\leqslant
			\||\yh|\partial_3(u_3\uh)(\tau)\|_{L^1}^{\frac{1}{p}}
			\||\yh|\partial_3(u_3\uh)(\tau)\|_{\Lh^1\Lv^{\infty}}^{1- \frac{1}{p}}\\
			&\leqslant
			CAB\tau^{- \frac{1}{p}}(1+\tau)^{-(1- \frac{1}{p})}.
		\end{split}
	\end{align}
	The mean value theorem yields
	\begin{align*}
		I_2(t,x)
		=-\int_0^{\frac{t}{2}}\int_{\mathbb{R}^2}\int_0^1(\nablah\Gh)(t,\xh-\theta\yh)\cdot\yh\partial_3(u_3\uh)(\tau,\yh,x_3)d\theta d\yh d\tau.
	\end{align*}
	Thus, we have by (\ref{I2modfyLh1Lvp})
	\begin{align*}
		\|I_2(t)\|_{L^p}
		&\leqslant
		\int_0^{\frac{t}{2}}\int_{\mathbb{R}^2}\int_0^1\|(\nablah\Gh)(t,\cdot-\theta\yh)\|_{L^p(\mathbb{R}^2)}\||\yh|\partial_3(u_3\uh)(\tau,\yh,\cdot)\|_{L^p(\mathbb{R})}d\theta d\yh d\tau\\
		&\leqslant
		C
		t^{-(1- \frac{1}{p})-\frac{1}{2}}
		\int_0^{\frac{t}{2}}\||\yh|\partial_3(u_3\uh)(\tau)\|_{\Lh^1\Lv^p} d\tau\\
		&\leqslant
		CAB
		t^{-(1- \frac{1}{p})-\frac{1}{2}}
		\int_0^{\frac{t}{2}}
		\tau^{- \frac{1}{p}}(1+\tau)^{-(1- \frac{1}{p})}d\tau\\
		&\leqslant
		CAB
		t^{-(1- \frac{1}{p})-\frac{1}{2}}
		\log(2+t)
	\end{align*}
	for $1<p\leqslant \infty$.
	This completes the proof.
\end{proof}
\begin{rem}\label{rem:Duha_ exp}
	From the above proof, we see that
	if the assumption {\bf (A3)} is modified and $\|\partial_3u(t)\|_{\Lh^{\infty}\Lv^1}$ is assumed to be bounded around $t=0$,
	then we obtain the estimate (\ref{Duha_expa_1-2}) even for the case $p=1$.
\end{rem}
\begin{lemm}\label{lemm:Duha_expa_2}
	Let $u$ satisfy {\bf (A1)} and {\bf (A2)} for some $s\geqslant 3$, $A>0$ and $T=\infty$.
	Then, for $1\leqslant p\leqslant \infty$, it holds
	\begin{align*}
			\begin{split}
				\lim_{t\to \infty}
				t^{(1- \frac{1}{p})+\frac{1}{2}}
				&\left\|
				\Nvv_1[u](t,x)-
				\nablah\Gh(t,\xh)\cdot
				\int_0^{\infty}
				\int_{\mathbb{R}^2}(u_3\uh)(\tau,\yh,x_3)d\yh d\tau
				\right\|_{L^p}
				=0.
			\end{split}
	\end{align*}
\end{lemm}
We omit the proof of Lemma \ref{lemm:Duha_expa_2}
since it is similar to that of (\ref{Duha_expa_1-1}).
\section{Proof of Theorem \ref{thm:main_1}}\label{s5}
In this section, we prove Theorem \ref{thm:main_1}.
First of all, we recall the global existence result for (\ref{eq:ANS_1}):
\begin{prop}[\cite{JWY}]\label{prop:existence}
	Let $s\in \mathbb{N}$ satisfy $s\geqslant 2$.
	Then, there exists a positive constant $\delta_0=\delta_0(s)$ such that
	for every $u_0\in H^s(\mathbb{R}^3)$ satisfying $\nabla \cdot u_0=0$ and $\|u_0\|_{H^s}\leqslant \delta_0$,
	(\ref{eq:ANS_1}) possesses a unique solution $u\in C([0,\infty);H^s(\mathbb{R}^3))$.
	Moreover, it holds
	\begin{align*}
		\|u(t)\|_{H^{\sigma}}^2+\int_0^t\|\nablah u(\tau)\|_{H^{\sigma}}^2 d\tau
		\leqslant 2\|u_0\|_{H^{\sigma}}^2, \qquad
		\sigma=0,1,...,s
	\end{align*}
	for all $t\geqslant 0$.
\end{prop}
We are ready to present the proof of Theorem \ref{thm:main_1}.
\begin{proof}[Proof of Theorem \ref{thm:main_1}]
	We split the proof into three steps.
	At the first step, we prove the solution belongs to $C([0,\infty);X^s(\mathbb{R}^3))$ if the $H^s$-norm of $u_0\in X^s(\mathbb{R}^3)$ is sufficiently small.
	In the second step, we show the $L^p$-decay estimate (\ref{main_1:uhL^p_decay}) and (\ref{main_1:u3L^p_decay}) by virture of Lemmas \ref{lemm:Linear_est_1}, \ref{lemm:Linear_est_2} and \ref{lemm:Duha_est_1}.
	Finally, the third step gives us the proof of asymptotic expansion by combining Lemmas \ref{lemm:Linear_est_1}, \ref{lemm:Linear_est_2}, \ref{lemm:Duha_est_1} and \ref{lemm:Duha_expa_1}.

	\noindent
	{\bf Step 1. (Global solutions in $X^s(\mathbb{R}^3)$)}

	Let $s\in \mathbb{N}$ satisfy $s\geqslant 5$
	and let $u_0\in X^s(\mathbb{R}^3)$ satisfy $\nabla\cdot u_0=0$ and $\|u_0\|_{H^s}\leqslant \delta_0$.
	Let $u$ be the global solution to (\ref{eq:ANS_1}), which is constructed by Proposition \ref{prop:existence}.
	From the integral equation (\ref{eq:Int-2}), we have
	\begin{align*}
		&\|u(t)\|_{L^1_{\xh}(W^{1,1}\cap W^{1,\infty})_{x_3}}\\
		&
		\quad
		\leqslant
		\|e^{t\Deltah}u_0\|_{L^1_{\xh}(W^{1,1}\cap W^{1,\infty})_{x_3}}
		+C\sum_{m=1}^5\|\Nhh_m[u](t)\|_{L^1_{\xh}W^{2,1}_{x_3}}
		+C\sum_{m=1}^3\|\Nvv_m[u](t)\|_{L^1_{\xh}W^{2,1}_{x_3}}\\
		&\quad
		\leqslant
		C\|u_0\|_{X^s}
		+C\int_0^t\|\partial_3(u_3\uh)(\tau)\|_{L^1_{\xh}W^{2,1}_{x_3}}d\tau
		+C\sum_{k,l=1}^3\int_0^t(t-\tau)^{-\frac{1}{2}}\|(u_ku_l)(\tau)\|_{L^1_{\xh}W^{2,1}_{x_3}}d\tau\\
		&\quad
		\leqslant
		C\|u_0\|_{X^s}
		+C\int_0^t\|u(\tau)\|_{H^3}^2d\tau
		+C\int_0^t(t-\tau)^{-\frac{1}{2}}\|u(\tau)\|_{H^2}^2d\tau\\
		&\quad\leqslant
		C\|u_0\|_{X^s}
		+Ct\|u_0\|_{H^3}^2
		+Ct^{\frac{1}{2}}\|u_0\|_{H^2}^2\\
		&\quad\leqslant
		C(1+t)\|u_0\|_{X^s}
	\end{align*}
	for all $t>0$.
	Here, we have used the abbreviation
	\begin{align*}
		&\|\cdot\|_{L^1_{\xh}(W^{1,1}\cap W^{1,\infty})_{x_3}}=\|\cdot\|_{L^1(\mathbb{R}^2_{\xh};(W^{1,1}\cap W^{1,\infty})(\mathbb{R}_{x_3}))},\\
		&\|\cdot\|_{L^1_{\xh}W^{2,1}_{x_3}}=\|\cdot\|_{L^1(\mathbb{R}^2_{\xh};W^{2,1}(\mathbb{R}_{x_3}))}.
	\end{align*}
	It is easy to check that $t\mapsto \|u(t)\|_{L^1_{\xh}(W^{1,1}\cap W^{1,\infty})_{x_3}}$ is continuous.
	Therefore, we have $u\in C([0,\infty);X^s(\mathbb{R}^3))$.

	\noindent
	{\bf Step 2. ($L^p$ decay estimates)}

	In this step, we show (\ref{main_1:uhL^p_decay}) and (\ref{main_1:u3L^p_decay}).
	The idea of the proof is the continuous argument used in \cite{JWY}.
	By Lemmas \ref{lemm:Linear_est_1} and \ref{lemm:Linear_est_2}, there exists an absolute positive constant $C_1$ such that
	\begin{align*}
		&\|\nabla^{\alpha}e^{t\Deltah}u_{0,\h}(t)\|_{L^p}
		\leqslant C_1t^{-(1-\frac{1}{p})-\frac{|\alphah|}{2}}\|u_0\|_{X^s},\\
		&\|\nablah^{\alphah}e^{t\Deltah}u_{0,3}(t)\|_{L^p}
		\leqslant C_1t^{-\frac{3}{2}(1-\frac{1}{p})-\frac{|\alphah|}{2}}\|u_0\|_{X^s},
	\end{align*}
	for all $t>0$ and $\alpha\in (\mathbb{N}\cup\{0\})^3$ with $|\alpha|\leqslant 1$.
	Let us consider
	\begin{align*}
		T:=
		\sup
		\left\{
			T'>0\ ;\
			\begin{array}{l}
				\|\nabla^{\alpha}\uh(t)\|_{L^p}
				\leqslant 2C_1t^{-(1-\frac{1}{p})-\frac{|\alphah|}{2}}\|u_0\|_{X^s},\\
				\|\nablah^{\alphah}u_3(t)\|_{L^p}
				\leqslant 2C_1t^{-\frac{3}{2}(1-\frac{1}{p})-\frac{|\alphah|}{2}}\|u_0\|_{X^s}
			\end{array}
			{\rm\ for\ }1\leqslant p \leqslant \infty{\rm \ and\ }0< t\leqslant T'
		\right\}.
	\end{align*}
	and we shall prove $T=\infty$.
	By the similar calculation as in the proof of $u(t)\in L^1(\mathbb{R}^2_{\xh};(W^{1,1}\cap W^{1,\infty})(\mathbb{R}_{x_3}))$,
	it is easy to see that $T>0$.
	Suppose by contradiction that $T<\infty$.
	Then, by Proposition \ref{prop:integral_equation} and Lemma \ref{lemm:Duha_est_1} with $A=2C_1\|u_0\|_{X^s}$,
	we see that
	\begin{align*}
		&\|\nabla^{\alpha}\uh(t)\|_{L^p}
		\leqslant C_1t^{-(1-\frac{1}{p})-\frac{|\alphah|}{2}}\|u_0\|_{X^s}+C_2(1+t)^{-(1-\frac{1}{p})-\frac{|\alphah|}{2}}\|u_0\|_{X^s}^2,\\
		&\|\nablah^{\alphah}u_3(t)\|_{L^p}
		\leqslant C_1t^{-\frac{3}{2}(1-\frac{1}{p})-\frac{|\alphah|}{2}}\|u_0\|_{X^s}+C_2(1+t)^{-\frac{3}{2}(1-\frac{1}{p})-\frac{|\alphah|}{2}}\|u_0\|_{X^s}^2
	\end{align*}
	for some $C_2>0$ and all $0<t<T$.
	Therefore, if the initial data satisfies
	\begin{align*}
		\|u_0\|_{X^s}\leqslant \delta_1:=\min\{\delta_0, C_1/(2C_2)\},
	\end{align*}
	then we have
	\begin{align*}
		&\|\nabla^{\alpha}\uh(t)\|_{L^p}
		\leqslant \frac{3}{2}C_1t^{-(1-\frac{1}{p})-\frac{|\alphah|}{2}}\|u_0\|_{X^s},\\
		&\|\nablah^{\alphah}u_3(t)\|_{L^p}
		\leqslant \frac{3}{2}C_1t^{-\frac{3}{2}(1-\frac{1}{p})-\frac{|\alphah|}{2}}\|u_0\|_{X^s}
	\end{align*}
	for $0<t<T$.
	This contradicts to the definition of $T$ and we complete the proof of (\ref{main_1:uhL^p_decay}) and (\ref{main_1:u3L^p_decay}).

	\noindent
	{\bf Step 3. (Asymptotic expansions)}

	By Lemmas \ref{lemm:Linear_est_1}, \ref{lemm:Duha_est_1} and \ref{lemm:Duha_expa_1}, we have
	\begin{align*}
		t^{1- \frac{1}{p}}
		&\left\|
			\uh(t,x)
			-
			\Gh(t,\xh)\int_{\mathbb{R}^2}u_{0,\h}(\yh,x_3)d\yh\right.\\
		&\left.
			\qquad \qquad+
			\Gh(t,\xh)\int_0^{\infty}\int_{\mathbb{R}^2}\partial_3(u_3\uh)(\tau,\yh,x_3)d\yh d\tau
		\right\|_{L^p_x}\\
		&\leqslant
		t^{1- \frac{1}{p}}
		\left\|e^{t\Deltah}u_{0,\h}(x)-\Gh(t,\xh)\int_{\mathbb{R}^2}u_{0,\h}(\yh,x_3)d\yh \right\|_{L^p_x}\\
		&\quad+
		t^{1- \frac{1}{p}}
		\left\|\Nhh_1[u](t)+\Gh(t,\xh)\int_0^{\infty}\int_{\mathbb{R}^2}\partial_3(u_3\uh)(\tau,\yh,x_3)d\yh d\tau\right\|_{L^p}
		+
		\sum_{m=2}^5t^{1- \frac{1}{p}}\|\Nhh_m[u](t)\|_{L^p}\\
		&\leqslant
		t^{1- \frac{1}{p}}
		\left\|e^{t\Deltah}u_{0,\h}(x)-\Gh(t,\xh)\int_{\mathbb{R}^2}u_{0,\h}(\yh,x_3)d\yh \right\|_{L^p_x}\\
		&\quad+
		t^{1- \frac{1}{p}}
		\left\|\Nhh_1[u](t)+\Gh(t,\xh)\int_0^{\infty}\int_{\mathbb{R}^2}\partial_3(u_3\uh)(\tau,\yh,x_3)d\yh d\tau\right\|_{L^p}
		+C\|u_0\|_{X^s}t^{-\frac{1}{2}}\log(2+t)\\
		&\to 0
		\qquad {\rm as\ }t\to \infty,
	\end{align*}
	which gives (\ref{main_1:asymp_uh}).
	For the proof of (\ref{main_1:asymp_u3}), we see by Lemmas \ref{lemm:Linear_est_2} and \ref{lemm:Duha_est_1} that
	\begin{align*}
		t^{\frac{3}{2}(1- \frac{1}{p})}
		&\left\|
			u_3(t,x)
			-
			\Gh(t,\xh)\int_{\mathbb{R}^2}u_{0,3}(\yh,x_3)d\yh
		\right\|_{L^p_x}\\
		&\leqslant
		t^{\frac{3}{2}(1- \frac{1}{p})}
		\left\|e^{t\Deltah}u_{0,3}(x)
		-
		\Gh(t,\xh)\int_{\mathbb{R}^2}u_{0,3}(\yh,x_3)d\yh
		\right\|_{L^p}
		+\sum_{m=1}^3t^{\frac{3}{2}(1- \frac{1}{p})}\|\Nvv_m[u](t)\|_{L^p}\\
		&\leqslant
		C\mathcal{R}_{p,0}(t)^{\frac{1}{p}}\|u_0\|_{L^1}^{1- \frac{1}{p}}
		+C\|u_0\|_{X^s}t^{-\frac{1}{2p}}+C\|u_0\|_{X^s}t^{-\frac{1}{2}+\frac{3}{8}(1- \frac{1}{p})}\log(2+t).
	\end{align*}
	Then, the right hand side converges to 0 as $t\to \infty$ if $p<\infty$.
	Thus, we complete the proof.
\end{proof}
\section{Additional Estimates for the Solution}\label{s6}
The aim of this section is to prove that the solution $u$ of (\ref{eq:ANS_1}) satisfies the assumptions {\bf (A3)} and {\bf (A4)} for $T=\infty$, $A=C\|u_0\|_{X^s}$ and $B=C\|u_0\|_{\widetilde{X^s}}$.
\begin{prop}\label{prop:Lh^infLv1-est}
	Let $s\in \mathbb{N}$ satisfy $s\geqslant 9$.
	Then, there exist a positive constant $\delta_3=\delta_3(s)\leqslant \delta_1(s)$ and an absolute positive constant $C$
	such that
	for solutions $u$ to (\ref{eq:ANS_1}) with the initial data $u_0\in X^s(\mathbb{R}^3)$ satisfying $\nabla\cdot u_0=0$ and $\|u_0\|_{X^s}\leqslant \delta_3$, it holds
	\begin{align*}
		\|\nabla^{\alpha}u(t)\|_{\Lh^{\infty}\Lv^1}
		\leqslant
		Ct^{-1- \frac{|\alphah|}{2}}\|u_0\|_{X^s}
	\end{align*}
	for all $t>0$ and $\alpha\in (\mathbb{N}\cup\{0\})^3$ with $|\alpha|\leqslant 1$.
\end{prop}
\begin{proof}
	By Lemma \ref{lemm:Linear_est_1}, there exists an absolute positive constant $C$ such that
	\begin{align*}
		\|\nabla^{\alpha}e^{t\Deltah}u_0\|_{\Lh^{\infty}\Lv^1}
		\leqslant
		Ct^{-1- \frac{|\alphah|}{2}}\|u_0\|_{X^s}
	\end{align*}
	for all $t>0$ and $\alpha\in (\mathbb{N}\cup\{0\})^3$ with $|\alpha|\leqslant 1$.
	Thus, the continuous argument via Lemma \ref{lemm:Duha_est_4} completes the proof.
	We omit the detail since the argument is similar to the second step of the proof of Theorem \ref{thm:main_1}.
\end{proof}
\begin{rem}\label{rem:Lh^infLv1-est}
	If we assume $\nabla^{\alpha} u_0\in\Lh^{\infty}\Lv^1(\mathbb{R}^3)$ for $|\alpha|\leqslant1$,
	then by the slight modification of the proof of Lemma \ref{lemm:Duha_est_4}
	and Proposition \ref{prop:Lh^infLv1-est}, we obtain
	\begin{align*}
		\|\nabla^{\alpha}u(t)\|_{\Lh^{\infty}\Lv^1}
		\leqslant
		C(\|u_0\|_{X^s}+\|\nabla^{\alpha}u_0\|_{\Lh^{\infty}\Lv^1})(1+t)^{-1-\frac{|\alphah|}{2}}
	\end{align*}
	for $t\geqslant 0$ and $|\alpha|\leqslant 1$.
	Hence by Remark \ref{rem:Duha_ exp}, (\ref{Duha_expa_1-2}) holds even when $p=1$, if $\nabla^{\alpha} u_0\in\Lh^{\infty}\Lv^1(\mathbb{R}^3)$ for $|\alpha|\leqslant1$.
\end{rem}

\begin{prop}\label{prop:weight}
	For $s\in \mathbb{N}$ with $s\geqslant 9$,
	there exist an absolute positive constant $C$ and a positive constant $\delta_2=\delta_2(s)$ with $\delta_2(s)\leqslant \delta_1(s)$ such that
	for any $u_0\in X^s(\mathbb{R}^3)$ satisfying $\nabla\cdot u_0=0$,
	$|\xh|u_0(x)\in L^1(\mathbb{R}^2_{\xh};(L^1 \cap L^{\infty})(\mathbb{R}_{x_3}))$
	and $\|u_0\|_{X^s}\leqslant \delta_2$, the solution $u$ of (\ref{eq:ANS_1}) satisfies
	\begin{align*}
		&\||\xh|\uh(t,x)\|_{\Lh^1\Lv^{\infty}}
		\leqslant
		C(1+t)^{\frac{1}{2}}\|u_0\|_{\widetilde{X^s}},\\
		&\||\xh|u_3(t,x)\|_{\Lh^1\Lv^{\infty}}
		\leqslant
		C\|u_0\|_{\widetilde{X^s}}
	\end{align*}
	for all $t\geqslant0$, where
	$\|u_0\|_{\widetilde{X^s}}:=\|u_0\|_{X^s}+\||\xh|u_0(x)\|_{L^1(\mathbb{R}^2_{\xh};(L^1 \cap L^{\infty})(\mathbb{R}_{x_3}))}$.
\end{prop}
To prove Proposition \ref{prop:weight}, let us consider the following integral equation:
\begin{align}\label{eq:linear_integral_eq}
	\begin{cases}
		\vh(t)
		=
		e^{t\Deltah}u_{0,\h}
		+\displaystyle\sum_{m=1}^5 \Nhh_m[u,v](t),\\
		v_3(t)
		=
		e^{t\Deltah}u_{0,3}
		+\displaystyle\sum_{m=1}^3\Nvv_m[u,v](t),
	\end{cases}
\end{align}
where
$v=(\vh(t,x),v_3(t,x))$ is the unknown vector field
and
$u$ is the solution to (\ref{eq:ANS_1}) with the initial data $u_0$ and the Duhamel terms are defined by
\begin{align*}
	\Nhh_1[u,v](t):&=-\int_0^te^{(t-\tau)\Deltah}((\partial_{3}u_3)\vh+v_3\partial_3\uh)(\tau)d\tau,\\
	\Nhh_2[u,v](t):&=-\int_0^te^{(t-\tau)\Deltah}\nablah \cdot (\uh\otimes\vh)(\tau)d\tau,\\
	\Nhh_3[u,v](t):&=\int_0^t\nablah e^{(t-\tau)\Deltah}(u_3v_3)(\tau)d\tau,\\
	\Nhh_4[u,v](t):&=-\sum_{k,l=1}^2\int_0^t\nablah\partial_{k}\partial_{l}K(t-\tau)*(u_kv_l)(\tau)d\tau,\\
	\Nhh_5[u,v](t):&=2\sum_{k=1}^2\int_0^t\nablah \partial_k(-\Deltah)^{\frac{1}{2}}\widetilde{K}(t-\tau)*(u_3v_k)(\tau)d\tau+\int_0^t\nablah \Deltah K(t-\tau)*(u_3v_3)(\tau)d\tau
\end{align*}
and
\begin{align*}
	\Nvv_1[u,v](t):&=\int_0^te^{(t-\tau)\Deltah}\nablah\cdot(\vh u_3)(\tau)d\tau,\\
	\Nvv_2[u,v](t):&=\sum_{k,l=1}^2\int_0^t(-\Deltah)^{\frac{1}{2}}\partial_{k}\partial_{l}\widetilde{K}(t-\tau)*(u_kv_l)(\tau)d\tau,\\
	\Nvv_3[u,v](t):&=2\sum_{k=1}^2\int_0^t\partial_k\Deltah K(t-\tau)*(u_3v_k)(\tau)d\tau+\int_0^t(-\Deltah)^{\frac{3}{2}}\widetilde{K}(t-\tau)*(u_3v_3)(\tau)d\tau.
\end{align*}

We note that $u$ is a solution to (\ref{eq:linear_integral_eq}).
Therefore, once we prove the existence of solutions to (\ref{eq:linear_integral_eq})
in an appropriate function space equipped with some weighted norm and the uniqueness in $L^{\infty}(0,\infty;\Lh^1\Lv^{\infty}(\mathbb{R}^3))$, then we obtain the desired weighted estimates for $u$.
Thanks to the definition of the Duhamel term $\Nhh_1[u,v]$,
the well-posedness problem for (\ref{eq:linear_integral_eq}) is much easier than that for (\ref{eq:ANS_1})
since no $\partial_3$-derivative loss occurs and we can construct the solution of (\ref{eq:linear_integral_eq}) simply by the contraction mapping approach.

We define the solution space $Y$ by
\begin{align*}
	Y:&=\left\{v=(v_1,v_2,v_3)\in L^{\infty}(0,\infty;\Lh^1\Lv^{\infty}(\mathbb{R}^3))\ ;\ \|v\|_{Y}<\infty \right\},\\
	\|v\|_Y
	:&=
	\sup_{t\geqslant 0}(1+t)^{-\frac{1}{2}}\||\xh|\vh(t,x)\|_{\Lh^1\Lv^{\infty}}
	+\sup_{t\geqslant 0}\||\xh|v_3(t,x)\|_{\Lh^1\Lv^{\infty}}\\
	&\quad
	+\sup_{t\geqslant 0}\|\vh(t)\|_{L^1(\mathbb{R}^2_{\xh};W^{1,\infty}(\mathbb{R}_{x_3}))}
	+\sup_{t\geqslant 0}(1+t)^{\frac{1}{2}}\|v_3(t)\|_{L^1(\mathbb{R}^2_{\xh};W^{1,\infty}(\mathbb{R}_{x_3}))}.
\end{align*}
We prepare some lemmas for the proof of Proposition \ref{prop:weight}.
\begin{lemm}\label{lemm:W_Duha_est_1}
	Let $s\in \mathbb{N}$ satisfy $s\geqslant 7$.
	Let $u$ be the solution to (\ref{eq:ANS_1}) with the data $u_0\in X^s(\mathbb{R}^3)$ satisfying $\nabla\cdot u_0=0$ and $\|u_0\|_{X^s}\leqslant \delta_1$.
	Then, there exists an absolute positive constant $C$ such that
	\begin{align}
		\||\xh|\Nhh_1[u,v](t,x)\|_{\Lh^1\Lv^{\infty}}
		&\leqslant
		C(1+t)^{\frac{1}{2}}\|u_0\|_{X^s}\|v\|_{Y},\label{W_Duha_est_1-1}\\
		\|\Nhh_1[u,v](t)\|_{\Lh^1\Lv^{\infty}}
		&\leqslant
		C\|u_0\|_{X^s}\|v\|_{L^{\infty}(0,\infty;\Lh^1\Lv^{\infty})},\label{W_Duha_est_1-3}\\
		\|\partial_3\Nhh_1[u,v](t)\|_{\Lh^1\Lv^{\infty}}
		&\leqslant
		C\|u_0\|_{X^s}\|v\|_{Y}\label{W_Duha_est_1-5}
	\end{align}
	for all $v\in Y$ and $t\geqslant 0$.
\end{lemm}
\begin{proof}
	For the proof of (\ref{W_Duha_est_1-1}), we have by Lemma \ref{lemm:Linear_est_1} (1) that
	\begin{align*}
		&\||\xh|\Nhh_1[u,v](t,x)\|_{\Lh^1\Lv^{\infty}}\\
		&\quad\leqslant
		\int_0^t \||\xh|\Gh(t-\tau,\xh)\|_{L^1(\mathbb{R}^2_{\xh})}
		\|((\partial_{3}u_3)\vh+v_3\partial_3\uh)(\tau)\|_{\Lh^1\Lv^{\infty}}d\tau\\
		&\qquad +
		\int_0^t \|\Gh(t-\tau)\|_{L^1(\mathbb{R}^2)}
		\|((\partial_{3}u_3)|\xh|\vh+|\xh|v_3\partial_3\uh)(\tau)\|_{\Lh^1\Lv^{\infty}}d\tau\\
		&\quad\leqslant
		C\int_0^t(t-\tau)^{\frac{1}{2}}
		\left(\|\partial_{3}u_3(\tau)\|_{L^{\infty}}\|\vh(\tau)\|_{\Lh^1\Lv^{\infty}}
		+\|v_3(\tau)\|_{\Lh^1\Lv^{\infty}}\|\partial_3\uh(\tau)\|_{L^{\infty}}\right)d\tau\\
		&\qquad +
		C\int_0^t
		\left(\|\partial_{3}u_3(\tau)\|_{L^{\infty}}\||\xh|\vh(\tau)\|_{\Lh^1\Lv^{\infty}}
		+\||\xh|v_3(\tau)\|_{\Lh^1\Lv^{\infty}}\|\partial_3\uh(\tau)\|_{L^{\infty}}\right)d\tau\\
		&\quad\leqslant
		C\int_0^t(t-\tau)^{\frac{1}{2}}(1+\tau)^{-\frac{3}{2}}d\tau\|u_0\|_{X^s}\|v\|_Y
		+
		C\int_0^t(1+\tau)^{-1}d\tau\|u_0\|_{X^s}\|v\|_Y\\
		&\quad \leqslant
		C(1+t)^{\frac{1}{2}}\|u_0\|_{X^s}\|v\|_Y.
	\end{align*}
	On the estimate for (\ref{W_Duha_est_1-3}), we have
	\begin{align*}
		\|\Nhh_1[u,v](t)\|_{\Lh^1\Lv^{\infty}}
		&\leqslant
		\int_0^t
		\|\Gh(t-\tau)\|_{L^1(\mathbb{R}^2)}
		\|((\partial_{3}u_3)\vh+v_3\partial_3\uh)(\tau)\|_{\Lh^1\Lv^{\infty}}d\tau\\
		&\leqslant
		C\int_0^t
		\left(\|\partial_{3}u_3(\tau)\|_{L^{\infty}}\|\vh(\tau)\|_{\Lh^1\Lv^{\infty}}
		+\|v_3(\tau)\|_{\Lh^1\Lv^{\infty}}\|\partial_3\uh(\tau)\|_{L^{\infty}}\right)d\tau\\
		&\leqslant
		C\int_0^t(1+\tau)^{-\frac{3}{2}}d\tau\|u_0\|_{X^s}\|v\|_{Y}\\
		&\leqslant
		C\|u_0\|_{X^s}\|v\|_{Y}.
	\end{align*}
	Finally, we prove (\ref{W_Duha_est_1-5}). It follows from (\ref{GW}) that
	\begin{align*}
		&\|\partial_3\Nhh_1[u,v](t)\|_{\Lh^1\Lv^{\infty}}\\
		&\quad\leqslant
		\int_0^t \|\Gh(t-\tau)\|_{L^1(\mathbb{R}^2)}
		\left(\|\partial_{3}^2u_3(\tau)\|_{L^{\infty}}\|\vh(\tau)\|_{\Lh^1\Lv^{\infty}}
		+\|\partial_3u_3(\tau)\|_{L^{\infty}}\|\partial_3\vh(\tau)\|_{\Lh^1\Lv^{\infty}}\right.\\
		&\quad \qquad \qquad \quad\left.
		+\|\partial_3v_3(\tau)\|_{\Lh^1\Lv^{\infty}}\|\partial_3\uh(\tau)\|_{L^{\infty}}
		+\|v_3(\tau)\|_{\Lh^1\Lv^{\infty}}\|\partial_3^2\uh(\tau)\|_{L^{\infty}}\right)d\tau\\
		&\quad\leqslant
		C\int_0^t
		\left(\|u_3(\tau)\|_{H^7}^{\frac{2}{7}}\|\partial_3u_3\|_{L^{\infty}}^{\frac{5}{7}}\|\vh(\tau)\|_{\Lh^1\Lv^{\infty}}
		+\|\partial_3u_3(\tau)\|_{L^{\infty}}\|\partial_3\vh(\tau)\|_{\Lh^1\Lv^{\infty}}\right.\\
		&\qquad \qquad \qquad \left.
		+\|\partial_3v_3(\tau)\|_{\Lh^1\Lv^{\infty}}\|\partial_3\uh(\tau)\|_{L^{\infty}}
		+\|v_3(\tau)\|_{\Lh^1\Lv^{\infty}}\|\uh(\tau)\|_{H^7}^{\frac{2}{7}}\|\partial_3\uh\|_{L^{\infty}}^{\frac{5}{7}}\right)d\tau\\
		&\quad\leqslant
		C\int_0^t(1+\tau)^{-\frac{15}{14}}d\tau\|u_0\|_{X^s}\|v\|_{Y}\\
		&\quad\leqslant
		C\|u_0\|_{X^s}\|v\|_{Y}.
	\end{align*}
	Thus, we complete the proof.
\end{proof}
\begin{lemm}\label{lemm:W_Duha_est_2}
	Let $s\in \mathbb{N}$ with $s\geqslant 5$ and let $u$ be the solution to (\ref{eq:ANS_1}) with the data $u_0\in X^s(\mathbb{R}^3)$ satisfying $\nabla\cdot u_0=0$ and $\|u_0\|_{X^s}\leqslant \delta_1$.
	Then, there exists an absolute positive constant $C$ such that
	\begin{align}
		&\left\||\xh|\int_0^te^{(t-\tau)\Deltah}\nablah(u_kv_l)(\tau,x)d\tau\right\|_{\Lh^1\Lv^{\infty}}
		\leqslant
		\begin{cases}
			C(1+t)^{\frac{1}{2}}\|u_0\|_{X^s}\|v\|_{Y} & (k=1,2)\\
			C\|u_0\|_{X^s}\|v\|_{Y} & (k=3)
		\end{cases}
		,\label{W_Duha_est_2-1}\\
		&\left\|\int_0^te^{(t-\tau)\Deltah}\nablah(u_kv_l)(\tau)d\tau\right\|_{\Lh^1\Lv^{\infty}}
		\leqslant
		\begin{cases}
			C\|u_0\|_{X^s}\|v\|_{L^{\infty}(0,\infty;\Lh^1\Lv^{\infty})} & (k=1,2)\\
			C(1+t)^{-\frac{1}{2}}\|u_0\|_{X^s}\|v\|_{L^{\infty}(0,\infty;\Lh^1\Lv^{\infty})} & (k=3)
		\end{cases}
		,\label{W_Duha_est_2-3}\\
		&\left\|\partial_3 \int_0^te^{(t-\tau)\Deltah}\nablah(u_kv_l)(\tau)d\tau\right\|_{\Lh^1\Lv^{\infty}}
		\leqslant
		\begin{cases}
			C\|u_0\|_{X^s}\|v\|_{Y} & (k=1,2)\\
			C(1+t)^{-\frac{1}{2}}\|u_0\|_{X^s}\|v\|_{Y} & (k=3)
		\end{cases}
		\label{W_Duha_est_2-4}
	\end{align}
	for all $t\geqslant 0$ and $l=1,2,3$.
\end{lemm}
\begin{proof}
	On the estimate (\ref{W_Duha_est_2-1}), we see by Lemma \ref{lemm:Linear_est_1} (1) that
	\begin{align*}
		&\left\||\xh|\int_0^te^{(t-\tau)\Deltah}\nablah(u_kv_l)(\tau,x)d\tau\right\|_{\Lh^1\Lv^{\infty}}\\
		&\quad\leqslant
		\int_0^t
		\||\xh|\nablah\Gh(t-\tau,\xh)\|_{L^1(\mathbb{R}^2_{\xh})}
		\|u_k(\tau)\|_{L^{\infty}}\|v_l(\tau)\|_{\Lh^1\Lv^{\infty}}d\tau\\
		&\qquad+
		\int_0^t
		\|\nablah\Gh(t-\tau)\|_{L^1(\mathbb{R}^2)}
		\|u_k(\tau)\|_{L^{\infty}}\||\xh|v_l(\tau)\|_{\Lh^1\Lv^{\infty}}d\tau\\
		&\quad\leqslant
		\begin{cases}
			C\displaystyle\int_0^t
			(1+\tau)^{-1}d\tau\|u_0\|_{X^s}\|v\|_Y & \\
			\quad+
			C\displaystyle\int_0^t
			(t-\tau)^{-\frac{1}{2}}
			(1+\tau)^{-\frac{1}{2}}d\tau\|u_0\|_{X^s}\|v\|_Y & (k=1,2)\\
			C\displaystyle\int_0^t
			(1+\tau)^{-\frac{3}{2}}d\tau\|u_0\|_{X^s}\|v\|_Y & \\
			\quad+
			C\displaystyle\int_0^t
			(t-\tau)^{-\frac{1}{2}}
			(1+\tau)^{-1}d\tau\|u_0\|_{X^s}\|v\|_Y & (k=3)\\
		\end{cases}\\
		&\quad\leqslant
		\begin{cases}
			C(1+t)^{\frac{1}{2}}\|u_0\|_{X^s}\|v\|_{Y} & (k=1,2)\\
			C\|u_0\|_{X^s}\|v\|_{Y} & (k=3)
		\end{cases}.
	\end{align*}
	We next obtain (\ref{W_Duha_est_2-3}) by
	\begin{align*}
		&\left\|\int_0^te^{(t-\tau)\Deltah}\nablah(u_kv_l)(\tau)d\tau\right\|_{\Lh^1\Lv^{\infty}}\\
		&\quad\leqslant
		C\int_0^t(t-\tau)^{-\frac{1}{2}}\|u_k(\tau)\|_{L^{\infty}}\|v_l(\tau)\|_{\Lh^1\Lv^{\infty}}d\tau\\
		&\quad\leqslant
		\begin{cases}
			C\displaystyle\int_0^t(t-\tau)^{-\frac{1}{2}}(1+\tau)^{-1}d\tau\|u_0\|_{X^s}\|v\|_{L^{\infty}(0,\infty;\Lh^1\Lv^{\infty})} & (k=1,2)\\
			C\displaystyle\int_0^t(t-\tau)^{-\frac{1}{2}}(1+\tau)^{-\frac{3}{2}}d\tau\|u_0\|_{X^s}\|v\|_{L^{\infty}(0,\infty;\Lh^1\Lv^{\infty})} & (k=3)
		\end{cases}\\
		&\quad\leqslant
		\begin{cases}
			C\|u_0\|_{X^s}\|v\|_{L^{\infty}(0,\infty;\Lh^1\Lv^{\infty})} & (k=1,2)\\
			C(1+t)^{-\frac{1}{2}}\|u_0\|_{X^s}\|v\|_{L^{\infty}(0,\infty;\Lh^1\Lv^{\infty})} & (k=3)
		\end{cases}.
	\end{align*}
	Finally, for the estimate (\ref{W_Duha_est_2-4}), we have
	\begin{align*}
		&\left\|\partial_3 \int_0^te^{(t-\tau)\Deltah}\nablah(u_kv_l)(\tau)d\tau\right\|_{\Lh^1\Lv^{\infty}}\\
		&\quad \leqslant
		\int_0^t\|\nablah\Gh(t-\tau)\|_{L^1(\mathbb{R}^2)}
		\left(
		\|\partial_3u_k(\tau)\|_{L^{\infty}}\|v_l(\tau)\|_{\Lh^1\Lv^{\infty}}
		+
		\|u_k(\tau)\|_{L^{\infty}}\|\partial_3v_l(\tau)\|_{\Lh^1\Lv^{\infty}}
		\right)d\tau\\
		&\quad\leqslant
		\begin{cases}
			C\displaystyle\int_0^t(t-\tau)^{-\frac{1}{2}}(1+\tau)^{-1}d\tau\|u_0\|_{X^s}\|v\|_{Y} & (k=1,2)\\
			C\displaystyle\int_0^t(t-\tau)^{-\frac{1}{2}}(1+\tau)^{-\frac{3}{2}}d\tau\|u_0\|_{X^s}\|v\|_{Y} & (k=3)
		\end{cases}\\
		&\quad\leqslant
		\begin{cases}
			C\|u_0\|_{X^s}\|v\|_{Y} & (k=1,2)\\
			C(1+t)^{-\frac{1}{2}}\|u_0\|_{X^s}\|v\|_{Y} & (k=3)
		\end{cases}.
	\end{align*}
	This completes the proof.
\end{proof}
\begin{lemm}\label{lemm:W_Duha_est_3}
		Let $s\in \mathbb{N}$ with $s\geqslant 9$ and let $u$ be the solution to (\ref{eq:ANS_1}) with the data $u_0\in X^s(\mathbb{R}^3)$ satisfying $\nabla\cdot u_0=0$ $\|u_0\|_{X^s}\leqslant \delta_3$.
		Then there exists an absolute positive constant $C$ such that
	\begin{align}
		&\left\||\xh|\int_0^tK^{(m)}_{\beta,\gamma}(t-\tau)*(u_kv_l)(\tau,x)d\tau\right\|_{\Lh^1\Lv^{\infty}}
		\leqslant
		C\|u_0\|_{X^s}\|v\|_{Y},\label{W_Duha_est_3-1}\\
		&\left\|\int_0^tK^{(m)}_{\beta,\gamma}(t-\tau)*(u_kv_l)(\tau)d\tau\right\|_{\Lh^1\Lv^{\infty}}
		\leqslant
		C(1+t)^{-\frac{1}{2}}\|u_0\|_{X^s}\|v\|_{L^{\infty}(0,\infty;\Lh^1\Lv^{\infty})},\label{W_Duha_est_3-3}\\
		&\left\|\partial_3 \int_0^tK^{(m)}_{\beta,\gamma}(t-\tau)*(u_kv_l)(\tau)d\tau\right\|_{\Lh^1\Lv^{\infty}}
		\leqslant
		C(1+t)^{-\frac{1}{2}}\|u_0\|_{X^s}\|v\|_{Y} \label{W_Duha_est_3-4}
	\end{align}
	for all $k,l=1,2,3$, $m=1,2$ and $t\geqslant 0$.
	where $K_{\beta,\gamma}^{(m)}$ are defined by (\ref{K^m})
	for $(\beta,\gamma)\in (\mathbb{N}\cup \{0\})^2\times (\mathbb{N}\cup \{0\})$ satisfying $|\beta|+\gamma=3$.
\end{lemm}
\begin{proof}
	For the estimate (\ref{W_Duha_est_3-1}), we have by Lemma \ref{lemm:Linear_est_3} that
	\begin{align*}
		&\left\||\xh|\int_0^tK^{(m)}_{\beta,\gamma}(t-\tau)*(u_kv_l)(\tau,x)d\tau\right\|_{\Lh^1\Lv^{\infty}}\\
		&\quad\leqslant
		\int_0^{t}
		\||\xh|K^{(m)}_{\beta,\gamma}(t-\tau,x)\|_{\Lh^1\Lv^{2}}\|u_k(\tau)v_l(\tau)\|_{\Lh^1\Lv^2}d\tau
		+
		\int_0^t
		\|K^{(m)}_{\beta,\gamma}(t-\tau)\|_{L^1}\|u_k(\tau)|\xh|v_l(\tau)\|_{\Lh^1\Lv^{\infty}}d\tau\\
		&\quad\leqslant
		C\int_0^t
		(t-\tau)^{-\frac{1}{4}}
		\|u_k(\tau)\|_{\Lh^{\infty}\Lv^1}^{\frac{1}{2}}\|u_k(\tau)\|_{L^{\infty}}^{\frac{1}{2}}
		\|v_l(\tau)\|_{\Lh^1\Lv^{\infty}}d\tau
		+
		C\int_0^t
		(t-\tau)^{-\frac{1}{2}}\|u_k(\tau)\|_{L^{\infty}}\||\xh|v_l(\tau)\|_{\Lh^1\Lv^{\infty}}d\tau\\
		&\quad\leqslant
		C\int_0^t(t-\tau)^{-\frac{1}{4}}\tau^{-\frac{1}{2}}(1+\tau)^{-\frac{1}{2}}d\tau\|u_0\|_{X^s}\|v\|_Y
		+
		\int_0^t
		(t-\tau)^{-\frac{1}{2}}(1+\tau)^{-\frac{1}{2}}d\tau\|u_0\|_{X^s}\|v\|_Y\\
		&\quad\leqslant
		C\|u_0\|_{X^s}\|v\|_Y.
	\end{align*}
	On the estimate (\ref{W_Duha_est_3-3}), we see that
	\begin{align*}
		&\left\|\int_0^tK^{(m)}_{\beta,\gamma}(t-\tau)*(u_kv_l)(\tau)d\tau\right\|_{\Lh^1\Lv^{\infty}}\\
		&\quad\leqslant
		\int_0^{\frac{t}{2}}
		\|K^{(m)}_{\beta,\gamma}(t-\tau)\|_{\Lh^1\Lv^{2}}\|u_k(\tau)v_l(\tau)\|_{\Lh^1\Lv^2}d\tau
		+
		\int_{\frac{t}{2}}^t
		\|K^{(m)}_{\beta,\gamma}(t-\tau)\|_{L^1}\|u_k(\tau)v_l(\tau)\|_{\Lh^1\Lv^{\infty}}d\tau\\
		&\quad\leqslant
		C\int_0^{\frac{t}{2}}
		(t-\tau)^{-\frac{3}{4}}
		\|u_k(\tau)\|_{\Lh^{\infty}\Lv^1}^{\frac{1}{2}}\|u_k(\tau)\|_{L^{\infty}}^{\frac{1}{2}}
		\|v_l(\tau)\|_{\Lh^1\Lv^{\infty}}d\tau
		+
		C\int_{\frac{t}{2}}^t
		(t-\tau)^{-\frac{1}{2}}\|u_k(\tau)\|_{L^{\infty}}\|v_l(\tau)\|_{\Lh^1\Lv^{\infty}}d\tau\\
		&\quad\leqslant
		C\int_0^{\frac{t}{2}}
		(t-\tau)^{-\frac{3}{4}}
		\tau^{-\frac{1}{2}}
		(1+\tau)^{-\frac{1}{2}}d\tau\|u_0\|_{X^s}\|v\|_{L^{\infty}(0,\infty;\Lh^1\Lv^{\infty})}
		+
		C\int_{\frac{t}{2}}^t
		(t-\tau)^{-\frac{1}{2}}(1+\tau)^{-1}d\tau\|u_0\|_{X^s}\|v\|_{L^{\infty}(0,\infty;\Lh^1\Lv^{\infty})}\\
		&\quad\leqslant
		CA^2(1+\tau)^{-\frac{1}{2}}\|u_0\|_{X^s}\|v\|_{L^{\infty}(0,\infty;\Lh^1\Lv^{\infty})}
	\end{align*}
	Finally, we can prove (\ref{W_Duha_est_3-4}).
	By the similar calculation as above, we see that
	\begin{align*}
		&\left\|\partial_3 \int_0^tK^{(m)}_{\beta,\gamma}(t-\tau)*(u_kv_l)(\tau)d\tau\right\|_{\Lh^1\Lv^{\infty}}\\
		&\quad\leqslant
		\int_0^{\frac{t}{2}}
		\|K^{(m)}_{\beta,\gamma}(t-\tau)\|_{\Lh^1\Lv^{2}}
		\left(\|\partial_3u_k(\tau)\|_{\Lh^{\infty}\Lv^2}\|v_l(\tau)\|_{\Lh^1\Lv^{\infty}}
		+
		\|u_k(\tau)\|_{\Lh^{\infty}\Lv^2}\|\partial_3v_l(\tau)\|_{\Lh^1\Lv^{\infty}}\right)
		d\tau\\
		&\qquad +
		\int_{\frac{t}{2}}^t
		\|K^{(m)}_{\beta,\gamma}(t-\tau)\|_{L^1}
		\left(\|\partial_3u_k(\tau)\|_{L^{\infty}}\|v_l(\tau)\|_{\Lh^1\Lv^{\infty}}
		+\|u_k(\tau)\|_{L^{\infty}}\|\partial_3v_l(\tau)\|_{\Lh^1\Lv^{\infty}}\right)
		d\tau\\
		&\quad\leqslant
		C(1+t)^{-\frac{1}{2}}\|u_0\|_{X^s}\|v\|_{Y}.
	\end{align*}
	Hence, we complete the proof.
\end{proof}

We are ready to prove Proposition \ref{prop:weight}.
\begin{proof}[Proof of Proposition \ref{prop:weight}]
	Let $s\in \mathbb{N}$ satisfy $s\geqslant 9$ and let $u_0\in X^s(\mathbb{R}^3)$ satisfy $\nabla \cdot u_0=0$ and $\|u_0\|_{X^s}\leqslant \delta_3$.
	Let $u$ be the solution to (\ref{eq:ANS_1}) with the initial data $u_0$.
	It is easy to see that there exists an absolute positive constant $C_4$ such that
	\begin{align*}
		\|e^{t\Deltah}u_0\|_{Y}\leqslant C_4\|u_0\|_{\widetilde{X^s}}.
	\end{align*}
	For each $v\in Y$, we define $\Psi[v]=(\Psi^{\h}[v],\Psi^{\vv}[v])$ by
	\begin{align*}
		\begin{cases}
			\Psi^{\h}[v](t):=e^{t\Deltah}u_{0,\h}+\displaystyle\sum_{m=1}^5\Nhh_m[u,v](t),\\
			\Psi^{\vv}[v](t):=e^{t\Deltah}u_{0,3}+\displaystyle\sum_{m=1}^3\Nvv_m[u,v](t).
		\end{cases}
	\end{align*}
	Then, by Lemmas \ref{lemm:W_Duha_est_1}, \ref{lemm:W_Duha_est_2} and \ref{lemm:W_Duha_est_3},
	there exists an absolute positive constant $C_5$ such that
	\begin{align*}
		&\left\|\left(\sum_{m=1}^5\Nhh_m[u,v],\sum_{m=1}^3\Nvv_m[u,v]\right)\right\|_{Y}
		\leqslant
		C_5\|u_0\|_{X^s}\|v\|_Y,\\
		&\left\|\left(\sum_{m=1}^5\Nhh_m[u,w],\sum_{m=1}^3\Nvv_m[u,w]\right)\right\|_{L^{\infty}(0,\infty;\Lh^1\Lv^{\infty})}
		\leqslant
		C_5\|u_0\|_{X^s}\|w\|_{L^{\infty}(0,\infty;\Lh^1\Lv^{\infty})}
	\end{align*}
	for all $v\in Y$ and $w\in L^{\infty}(0,\infty;\Lh^1\Lv^{\infty}(\mathbb{R}^3))$.
	Thus, it holds
	\begin{align}\label{vY}
		\begin{split}
			\|\Psi[v]\|_Y
				&\leqslant \|e^{t\Deltah}u_0\|_Y+\left\|\left(\sum_{m=1}^5\Nhh_m[u,v],\sum_{m=1}^3\Nvv_m[u,v]\right)\right\|_{Y}\\
				&\leqslant
				C_4\|u_0\|_{\widetilde{X^s}}
				+C_5\|u_0\|_{X^s}\|v\|_Y<\infty
		\end{split}
	\end{align}
	for $v\in Y$, which implies $\Psi[v]\in Y$.
	Let $\delta_2=\min\{\delta_3,1/(2C_5)\}$ and assume $\|u_0\|_{X^s}\leqslant \delta_2$.
	Then, we have
	\begin{align*}
		\|\Psi[v_1]-\Psi[v_2]\|_Y
			&\leqslant \left\|\left(\sum_{m=1}^5\Nhh_m[u,v_1-v_2],\sum_{m=1}^3\Nvv_m[u,v_1-v_2]\right)\right\|_{Y}\\
			&\leqslant C_5\|u_0\|_{X^s}\|v_1-v_2\|_Y\\
			&\leqslant \frac{1}{2}\|v_1-v_2\|_Y.
	\end{align*}
	Therefore, the contraction mapping principle yields that
	there exists a unique $\widetilde{v}\in Y$ such that $\widetilde{v}=\Psi[\widetilde{v}]$.
	Then by (\ref{vY}), we have
	\begin{equation*}
		\|\widetilde{v}\|_Y
			\leqslant
			C_4\|u_0\|_{\widetilde{X^s}}
			+\frac{1}{2}\|\widetilde{v}\|_Y,
	\end{equation*}
	which implies $\|\widetilde{v}\|_Y\leqslant 2C_4\|u_0\|_{\widetilde{X^s}}$.

	Finally, we show $\widetilde{v}=u$.
	Since $\widetilde{v}, u\in L^{\infty}(0,\infty;\Lh^1\Lv^{\infty}(\mathbb{R}^3))$ and $u$ is also a solution to (\ref{eq:linear_integral_eq}),
	we have
	\begin{align*}
		\begin{split}
			\|\widetilde{v}-u\|_{L^{\infty}(0,\infty;\Lh^1\Lv^{\infty})}
				&=\|\Psi[\widetilde{v}]-\Psi[u]\|_{L^{\infty}(0,\infty;\Lh^1\Lv^{\infty})}\\
				&\leqslant C_5\|u_0\|_{X^s}\|\widetilde{v}-u\|_{L^{\infty}(0,\infty;\Lh^1\Lv^{\infty})}\\
				&\leqslant \frac{1}{2}\|\widetilde{v}-u\|_{L^{\infty}(0,\infty;\Lh^1\Lv^{\infty})}.
		\end{split}
	\end{align*}
	This implies $\widetilde{v}=u$.
	Hence, we see that $u\in Y$ and $\|u\|_Y\leqslant 2C_4\|u_0\|_{\widetilde{X^s}}$, which complete the proof.
\end{proof}
\section{Proofs of Theorem \ref{thm:main_2} and Corollary \ref{cor}}\label{s7}
Now, we give the proofs of Theorem \ref{thm:main_2}  and Corollary \ref{cor}.
\begin{proof}[Proof of Proposition \ref{thm:main_2}]
	Let $s\in \mathbb{N}$ satisfy $s\geqslant 9$ and let $\delta_2$ be the constant determined in Proposition \ref{prop:weight}.
	Let $u_0\in X^s(\mathbb{R}^3)$ satisfy $|\xh|u_0(x)\in L^1(\mathbb{R}^2_{\xh};(L^1\cap L^{\infty})(\mathbb{R}_{x_3}))$, $\nabla\cdot u_0=0$ and $\|u_0\|_{X^s}\leqslant \delta_2$.
	Then, by Theorem \ref{thm:main_1}, Proposition \ref{prop:Lh^infLv1-est} and Proposition \ref{prop:weight}, the solution $u$ of (\ref{eq:ANS_1}) with the initial data $u_0$ satisfies the assumptions {\bf (A1)}-{\bf (A4)} with $T=\infty$,
	$A=C\|u_0\|_{X^s}$ and $B=C\|u_0\|_{\widetilde{X^s}}$ for some constant $C$.
	Hence, Lemmas \ref{lemm:Linear_est_1}, \ref{lemm:Duha_est_1} and \ref{lemm:Duha_expa_1} imply that
	\begin{align*}
		&\left\|
			\uh(t,x)
			-
			\Gh(t,\xh)\int_{\mathbb{R}^2}u_{0,\h}(\yh,x_3)d\yh\right.\\
		&\left.
			\qquad \qquad+
			\Gh(t,\xh)\int_0^{\infty}\int_{\mathbb{R}^2}\partial_3(u_3\uh)(\tau,\yh,x_3)d\yh d\tau
		\right\|_{L^p_x}\\
		&\quad\leqslant
		\left\|e^{t\Deltah}u_{0,\h}(x)-\Gh(t,\xh)\int_{\mathbb{R}^2}u_{0,\h}(\yh,x_3)d\yh \right\|_{L^p_x}\\
		&\qquad+
		\left\|\Nhh_1[u](t)+\Gh(t,\xh)\int_0^{\infty}\int_{\mathbb{R}^2}\partial_3(u_3\uh)(\tau,\yh,x_3)d\yh d\tau\right\|_{L^p_x}
		+
		\sum_{m=2}^5\|\Nhh_m[u](t)\|_{L^p_x}\\
		&\quad\leqslant
		C\||\xh|u_{0,\h}\|_{\Lh^1\Lv^p}t^{-(1- \frac{1}{p})-\frac{1}{2}}
		+C\|u_0\|_{X^s}\|u_0\|_{\widetilde{X^s}}t^{-(1- \frac{1}{p})-\frac{1}{2}}\log(2+t)\\
		&\quad \leqslant
		C\|u_0\|_{\widetilde{X^s}}t^{-(1- \frac{1}{p})-\frac{1}{2}}\log t
	\end{align*}
	for $t\geqslant 2$.
	By Lemmas \ref{lemm:Linear_est_1} and \ref{lemm:Duha_est_1}, we have
	\begin{align*}
		&\left\|
			u_3(t,x)
			-
			\Gh(t,\xh)\int_{\mathbb{R}^2}u_{0,3}(\yh,x_3)d\yh
		\right\|_{L^p_x}\\
		&\quad\leqslant
		t^{\frac{3}{2}(1- \frac{1}{p})}
		\left\|e^{t\Deltah}u_{0,3}(x)
		-
		\Gh(t,\xh)\int_{\mathbb{R}^2}u_{0,3}(\yh,x_3)d\yh
		\right\|_{L^p_x}
		+\sum_{m=1}^3\|\Nvv_m[u](t)\|_{L^p_x}\\
		&\quad\leqslant
		Ct^{-\frac{3}{2}(1- \frac{1}{p})-\frac{1}{2p}}\||\xh|u_{0}(x)\|_{L^1_x}
		+C\|u_0\|_{{X^s}}^2t^{-\frac{3}{2}(1- \frac{1}{p})-\frac{1}{2p}}
		+C\|u_0\|_{{X^s}}^2t^{-\frac{9}{8}(1- \frac{1}{p})-\frac{1}{2}}\log(2+t)\\
		&\quad\leqslant
		\begin{cases}
			C\|u_0\|_{\widetilde{X^s}}t^{-\frac{3}{2}(1- \frac{1}{p})- \frac{1}{2p}} & (1<p\leqslant \infty)\\
			C\|u_0\|_{\widetilde{X^s}}t^{- \frac{1}{2}}\log t & (p=1)
		\end{cases}
	\end{align*}
	for $t\geqslant 2$.
	Finally, we show (\ref{main_2:asymp_u3_second}).
	By Lemmas \ref{lemm:Linear_est_2}, \ref{lemm:Duha_est_1} and \ref{lemm:Duha_expa_2}, we obtain
	\begin{align*}
		&t^{\frac{3}{2}(1- \frac{1}{p})+\frac{1}{2p}}
		\left\|
		u_3(t,x)
		-\Gh(t,\xh)\int_{\mathbb{R}^2}u_{0,3}(\yh,x_3)d\yh
		+\nablah \Gh(t,\xh)\cdot \int_{\mathbb{R}^2}\yh u_{0,3}(\yh,x_3)d\yh
		\right.\\
		&\left. \qquad \qquad \qquad \qquad
		-\nablah \Gh(t,\xh) \cdot \int_0^{\infty}\int_{\mathbb{R}^2}(u_3\uh)(\tau,\yh,x_3)d\yh d\tau
		\right\|_{L^p_x}\\
		&\quad\leqslant
		t^{\frac{3}{2}(1- \frac{1}{p})+\frac{1}{2p}}
		\left\|
		e^{t\Deltah}u_{0,3}(x)
		-\Gh(t,\xh)\int_{\mathbb{R}^2}u_{0,3}(\yh,x_3)d\yh
		+\nablah \Gh(t,\xh)\cdot \int_{\mathbb{R}^2}\yh u_{0,3}(\yh,x_3)d\yh
		\right\|_{L^p_x}\\
		&\qquad+
		t^{\frac{3}{2}(1- \frac{1}{p})+\frac{1}{2p}}
		\left\|
		\Nvv_1[u](t,x)
		-\nablah \Gh(t,\xh) \cdot \int_0^{\infty}\int_{\mathbb{R}^2}(u_3\uh)(\tau,\yh,x_3)d\yh d\tau
		\right\|_{L^p_x}\\
		&\qquad+
		t^{\frac{3}{2}(1- \frac{1}{p})+\frac{1}{2p}}\|\Nvv_2[u](t)\|_{L^p}
		+t^{\frac{3}{2}(1- \frac{1}{p})+\frac{1}{2p}}\|\Nvv_3[u](t)\|_{L^p}\\
		&\quad\leqslant
		Ct^{-\frac{1}{2}+\frac{1}{2p}}\mathcal{R}_{p,1}(t)^{\frac{1}{p}}\||\xh|u_0\|_{L^1}^{1- \frac{1}{p}}\\
		&\qquad+
		t^{(1- \frac{1}{p})+ \frac{1}{2}}
		\left\|
		\Nvv_1[u](t,x)
		-\nablah \Gh(t,\xh) \cdot \int_0^{\infty}\int_{\mathbb{R}^2}(u_3\uh)(\tau,\yh,x_3)d\yh d\tau
		\right\|_{L^p_x}\\
		&\qquad+
		Ct^{-\frac{1}{8}(1-\frac{1}{p})}\log(2+t).
	\end{align*}
	Then, the right hand side converges to $0$ if $1<p\leqslant \infty$.
	Hence, we complete the proof.
\end{proof}
\begin{proof}[Proof of Corollary \ref{cor}]
	Passing the limit inferior as $t\to \infty$ in
	\begin{align*}
		&t^{\frac{3}{2}}\left\|
		u_3(t,x)
		-\Gh(t,\xh)\int_{\mathbb{R}^2}u_{0,3}(\yh,x_3)d\yh
		\right\|_{L^{\infty}_x}\\
		&\quad\geqslant
		t^{\frac{3}{2}}
		\left\|
		\nablah \Gh(t,\xh)\cdot \int_{\mathbb{R}^2}\yh u_{0,3}(\yh,x_3)d\yh
		-\nablah \Gh(t,\xh) \cdot \int_0^{\infty}\int_{\mathbb{R}^2}(u_3\uh)(\tau,\yh,x_3)d\yh d\tau
		\right\|_{L^{\infty}_x}\\
		&\qquad -
		t^{\frac{3}{2}}
		\left\|
		u_3(t,x)
		-\Gh(t,\xh)\int_{\mathbb{R}^2}u_{0,3}(\yh,x_3)d\yh
		+\nablah \Gh(t,\xh)\cdot \int_{\mathbb{R}^2}\yh u_{0,3}(\yh,x_3)d\yh
		\right.\\
		&\left. \qquad \qquad \qquad \qquad
		-\nablah \Gh(t,\xh) \cdot \int_0^{\infty}\int_{\mathbb{R}^2}(u_3\uh)(\tau,\yh,x_3)d\yh d\tau
		\right\|_{L^{\infty}_x},
	\end{align*}
	we have by (\ref{main_2:asymp_u3_second}) that
	\begin{align*}
		&\liminf_{t\to \infty}
		t^{\frac{3}{2}}\left\|
		u_3(t,x)
		-\Gh(t,\xh)\int_{\mathbb{R}^2}u_{0,3}(\yh,x_3)d\yh
		\right\|_{L^{\infty}_x}\\
		&\quad\geqslant
		\liminf_{t\to \infty}
		t^{\frac{3}{2}}
		\left\|
		\nablah \Gh(t,\xh)\cdot \int_{\mathbb{R}^2}\yh u_{0,3}(\yh,x_3)d\yh
		-\nablah \Gh(t,\xh) \cdot \int_0^{\infty}\int_{\mathbb{R}^2}(u_3\uh)(\tau,\yh,x_3)d\yh d\tau
		\right\|_{L^{\infty}_x}.
	\end{align*}
	Here, let us consider the following initial data:
	\begin{align*}
		u_0(x):=\eta\phi(x),\qquad
		\phi(x):=\left(0,-x_3e^{-|x|^2},x_2e^{-|x|^2}\right),
	\end{align*}
	where $\eta\in (0,\delta_2(9)/\|\phi\|_{X^{9}}]$ is a positive constant to be determined later.
	It is easy to check that this $u_0$ satisfies
	$u_0\in X^{9}(\mathbb{R}^3)$, $|\xh|u_0(x)\in L^1(\mathbb{R}^2_{\xh};(L^1\cap L^{\infty})(\mathbb{R}_{x_3}))$,
	$\nabla\cdot u_0=0$
	and $\|u_0\|_{X^{9}}\leqslant \delta_2(9)$.
	Then, we see that
	\begin{align*}
		\nablah \Gh(t,\xh)\cdot \int_{\mathbb{R}^2}\yh u_{0,3}(\yh,x_3)d\yh
		&=
		\eta\partial_1\Gh(t,\xh)\int_{-\infty}^{\infty}y_1e^{-y_1^2}dy_1\int_{-\infty}^{\infty}y_2e^{-y_2^2}dy_2e^{-x_3^2}\\
		&\quad
		+\eta\partial_2\Gh(t,\xh)\int_{-\infty}^{\infty}e^{-y_1^2}dy_1\int_{-\infty}^{\infty}y_2^2e^{-y_2^2}dy_2e^{-x_3^2}\\
		&=C\eta t^{-\frac{3}{2}}\partial_2\Gh(1,t^{-\frac{1}{2}}\xh)e^{-x_3^2},
	\end{align*}
	which implies
	\begin{align*}
		\left\|
		\nablah \Gh(t,\xh)\cdot \int_{\mathbb{R}^2}\yh u_{0,3}(\yh,x_3)d\yh
		\right\|_{L^{\infty}_x}
		=C_6\eta t^{-\frac{3}{2}}
	\end{align*}
	for some absolute positive constant $C_6$.
	On the other hand, the corresponding solution $u$ satisfies
	\begin{align*}
		\left\|
		\nablah \Gh(t,\xh) \cdot \int_0^{\infty}\int_{\mathbb{R}^2}(u_3\uh)(\tau,\yh,x_3)d\yh d\tau
		\right\|_{L^{\infty}_x}
		&\leqslant
		Ct^{-\frac{3}{2}}\int_0^{\infty}\|u_3(\tau)\|_{L^{\infty}}\|\uh(\tau)\|_{\Lh^1\Lv^{\infty}}d\tau\\
		&\leqslant
		C\int_0^{\infty}(1+\tau)^{-\frac{3}{2}}d\tau\cdot\eta^2\|\phi\|_{X^{9}}^2t^{-\frac{3}{2}}\\
		&=
		C_7\eta^2t^{-\frac{3}{2}}
	\end{align*}
	for some absolute positive constant $C_7$.
	Hence, if we choose $\eta$ such that $0<\eta\leqslant \min\{\delta_2(9)/\|\phi\|_{X^{9}},C_6/(2C_7)\}$, then we have
	\begin{align*}
		\liminf_{t\to \infty}
		t^{\frac{3}{2}}\left\|
		u_3(t,x)
		-\Gh(t,\xh)\int_{\mathbb{R}^2}u_{0,3}(\yh,x_3)d\yh
		\right\|_{L^{\infty}_x}
		\geqslant
		C_6\eta-C_7\eta^2
		\geqslant
		\frac{C_6}{2}\eta
		>0.
	\end{align*}
	This completes the proof.
\end{proof}

\noindent
{\bf Acknowledgements.} \\
This work was partly supported by Grant-in-Aid for JSPS Research Fellow, Grant Number JP20J20941.
The author would like to express his sincere gratitude to Professor Jun-ichi Segata, Faculty of Mathematics, Kyushu University, for many fruitful advices and continuous encouragement.

\end{document}